\documentclass{article}

\usepackage{natbib,hyperref}
\usepackage{proof,stmaryrd,amsmath,amssymb,amsthm,authblk}

\usepackage{xcolor}
\hypersetup{
    colorlinks,
    linkcolor={red!50!black},
    citecolor={blue!50!black},
    urlcolor={blue!80!black}
}

\usepackage[matrix,arrow]{xy}

\newcommand{\Tensor}{\otimes}
\newcommand{\Par}{\mathop{\bindnasrepma}}
\newcommand{\AND}{\mathop{\binampersand}}
\newcommand{\OR}{\oplus}

\newcommand{\BS}{\mathop{\backslash}}
\newcommand{\SL}{\mathop{/}}
	
\newcommand{\np}{\bar{p}}

\newcommand{\Wc}{\mathcal{W}}
\newcommand{\Cc}{\mathcal{C}}
\newcommand{\Ec}{\mathcal{E}}
\newcommand{\Ic}{\mathcal{I}}

\newcommand{\CUT}{\mathrm{cut}}
\newcommand{\MIX}{\mathrm{mix}}
\newcommand{\CONTR}{\mathrm{contr}}
\newcommand{\NCONTR}{\mathrm{ncontr}}
\newcommand{\EXC}{\mathrm{ex}}
\newcommand{\WEAK}{\mathrm{weak}}
\newcommand{\AX}{\mathrm{ax}}

\newcommand{\TRANSFORM}{\qquad\mbox{\raisebox{1em}{\Large$\leadsto$}}\qquad}

\newcommand{\U}{\mathbf{1}}
\newcommand{\Z}{\mathbf{0}}

\newcommand{\Af}{\mathfrak{A}}

\newcommand{\Lambek}{\mathrm{SLC}_\Sigma}
\newcommand{\LambekA}{\mathrm{SMALC}_\Sigma}
\newcommand{\LambekE}{\mathrm{ELC}}
\newcommand{\Linlog}{\mathrm{SCLL}_\Sigma}
\newcommand{\SMCLL}{\mathrm{SMCLL}_\Sigma}

\newcommand{\Var}{\mathrm{Var}}

\newtheorem{theorem}{Theorem}
\newtheorem{corollary}[theorem]{Corollary}
\newtheorem{lemma}[theorem]{Lemma}

\begin{document}

\title{Subexponentials in Non-Commutative Linear Logic}

\renewcommand{\Affilfont}{\small}

\author[1,5]{Max Kanovich}
\author[2,5]{Stepan Kuznetsov}
\author[3]{Vivek Nigam}
\author[4,5]{\mbox{Andre Scedrov}}

\date{}

\affil[1]{University College London, UK}
\affil[2]{Steklov Mathematical Institute of the RAS, Moscow, Russia}
\affil[3]{Fortiss GmbH, Munich, Germany}
\affil[4]{University of Pennsylvania, Philadelphia, USA}
\affil[5]{National Research University Higher School of Economics, Moscow, Russia}


\maketitle

\begin{abstract}
Linear logical frameworks with subexponentials have been used for the specification of among other
systems, proof systems, concurrent programming languages and linear authorization logics. In these
frameworks, subexponentials can be configured to allow or not for the application of the contraction
and weakening rules while the exchange rule can always be applied. This means that formulae in such
frameworks can only be organized as sets and multisets of formulae not being possible to organize
formulae as lists of formulae. This paper investigates the proof theory of linear logic proof systems in
the non-commutative variant. These systems can disallow the application of exchange rule on some
subexponentials. We investigate conditions for when cut-elimination is admissible in the presence of
non-commutative subexponentials, investigating the interaction of the exchange rule with local and
non-local contraction rules. We also obtain some new undecidability and decidability results on non-commutative linear
logic with subexponentials.
\end{abstract}

\vskip 2em 
\begin{center}
\it To Dale Miller's Festschrift and his Contributions to Logic in Computer Science. Dale's work
has been an inspiration to us. He is a great researcher, colleague, advisor, and friend.
\end{center}

\section{Introduction}\label{S:intro} 

Logic and proof theory have played an important role in computer science. The introduction of linear logic by~\citet{Girard} is an example of how the beauty of logic can be applied to the principles of computer science. More than 20 years ago,~\citet{hodas91lics,hodas94ic} proposed the intuitionistic linear logical framework, Lolli, which distinguishes between to kinds of formulae: \emph{linear}, that cannot be contracted and weakened, and \emph{unbounded}, that can be contracted and weakened\footnote{The authors received the LICS Test of Time Award for this work.}. In contrast to existing intuitionistic/classical logical frameworks, Lolli allowed to express stateful computations using logical connectives. Some years later, Miller proposed the classical linear logical framework Forum~\citep{miller94lics,miller96tcs} demonstrating that linear logic can be used among other things to design concurrent systems\footnote{For this work, Miller received yet another LICS Test of Time Award prize.}.


It has been known, however,  since Girard's original linear logic paper~\citep{Girard}, that the linear logic exponentials $!,?$ are not canonical.
Indeed, proof systems with non-equivalent exponentials~\citep{danos93kgc} can be formulated. \citet{nigam09ppdp} called them subexponentials and proposed a more expressive linear logical framework called SELL which allows for the specification of any number of non-equivalent subexponentials ${!}^s,{?}^s$. Each subexponential can be specified to behave as linear or as unbounded. This is reflected in the syntax. SELL sequents associate a different context to each subexponential. Thus formulae may be organized into a number of sets of unbounded formulae and a number of multisets of linear formulae. Nigam and Miller show that SELL is more expressive than Forum being capable of expressing algorithmic specifications in logic. In the recent years, it has been shown that SELL can also be used to specify linear authorization logics~\citep{nigam12lics,nigam14tcs}, concurrent constraint programming languages~\citep{nigam13concur,olarte15tcs} and proof systems~\citep{nigam16jlc}.

While these logical frameworks have been sucessfully used for a number of applications, they do not allow formulae to be organized as lists of formulae. This is because all the frameworks above assume that the exchange rule can be applied to any formula. This paper investigates the proof theory of 
subexponentials in non-commutative linear logic.
Our contribution is as follows:

%

\begin{enumerate}
\item We construct general non-commutative linear logic proof systems with subexponentials and investigate conditions for when these systems
enjoy cut-elimination and when they don't.
\item For systems, in which at least one subexponential obeys the contraction rule in its non-local form, we prove undecidability results.
\item For fragments, in which no subexponential obeys the contraction rule,
we prove decidability and establish exact complexity bounds which coincide with the complexity estimations for the corresponding systems without subexponentials:
NP for the purely multiplicative system, PSPACE for the system with additive connectives.
\end{enumerate}


The rest of this paper is organised as follows. In Sections~\ref{S:Lambek} and~\ref{S:Linlog} we
present two variants of non-commutative linear logic, resp., the multiplicative-additive Lambek calculus 
($\LambekA$) and
cyclic linear logic ($\Linlog$), enriched with subexponential modalities indexed by a subexponential
signature $\Sigma$.  In Section~\ref{S:cutelim} we establish the cut elimination property for
$\Linlog$ using the classical Gentzen's approach with a specific version of the mix rule.
In Section~\ref{S:embed} we show that $\LambekA$ can be conservatively embedded into $\Linlog$. This 
yields, as a side-effect, cut elimination for $\LambekA$. In Section~\ref{S:CutVsContr} we explain why we prefer the {\em non-local} version
of the contraction rule by showing that systems with only local contraction fail to enjoy the cut
elimination property. Section~\ref{S:undec} contains the proof of undecidability for systems
with contraction; in Section~\ref{S:decid} we prove decidability and establish complexity bounds for systems
without contraction. Section~\ref{S:future} is for conclusions and directions of future research.

\section{The Multiplicative-Additive Lambek Calculus with Subexponentials}\label{S:Lambek}

We start with the Lambek calculus allowing empty antecedents~\citep{Lambek61}, considering it
as a non-commutative form of intuitionistic propositional linear logic~\citep{Abrusci}.
The original Lambek calculus includes only multiplicative connectives (multiplication and
two implications, called divisions). It is quite natural, however, to equip the Lambek calculus
also with additive connectives (conjunction and disjunction), as in linear 
logic~\citep{vanBenthemLIA,KanazawaJoLLI,BuszkoLCSL,KuznetsovOkhotin}. We'll call this
bigger system the {\em multiplicative-additive Lambek calculus} (MALC).
Extended versions of the Lambek calculus have broad linguistical applications, serving as 
a basis for categorial grammars~\citep{MoortgatHandbook,MorrillBook,MorrillPhilosophy,MootRetore}.

In this section we extend the multiplicative-additive Lambek calculus 
with a family of subexponential connectives.
First we fix a {\em subexponential signature} of the form
$$
\Sigma = \langle \mathcal{I}, \preceq, \mathcal{W}, \mathcal{C}, \mathcal{E} \rangle,
$$
where $\mathcal{I} = \{ s_1, \ldots, s_n \}$ is a set of subexponential labels with a preorder $\preceq$,
and $\mathcal{W}$, $\mathcal{C}$, and $\mathcal{E}$ are subsets of $\mathcal{I}$.
The sets $\mathcal{W}$, $\mathcal{C}$, and $\mathcal{E}$ are required to be upwardly closed with respect to $\preceq$. That is, if $s_1 \in \mathcal{W}$ and $s_1 \preceq s_2$, then $s_2 \in \mathcal{W}$ and
   ditto for the sets $\mathcal{E}$ and $\mathcal{C}$. Subexponentials
   marked with labels from $\Wc$ allow weakening, $\Cc$ allows contraction,
   and $\Ec$ allows exchange (permutation). Since contraction (in the non-local form,
   see below) and weakening yield exchange, here we explicitly require 
   $\Wc \cap \Cc \subseteq \Ec$.

Formulae are built from variables $p_1, p_2, p_3, \ldots$ and the unit
constant $\U$ using five binary connectives: $\cdot$ (product, or multiplicative 
conjunction), $\BS$ (left division), $\SL$ (right division),
$\wedge$ (additive conjunction), and $\vee$ (additive disjunction), and a family
of unary connectives, indexed by the subexponential signature $\Sigma$, denoted
by ${!}^s$ for each $s \in \Ic$.

The axioms and rules of the multiplicative-additive Lambek calculus with
subexponentials, denoted by $\LambekA$, are as follows:

$$
\infer[(\AX)]{A \to A}{}
$$

$$
\infer[(\cdot\to)]{\Gamma_1, A \cdot B, \Gamma_2 \to C}{\Gamma_1, A, B, \Gamma_2 \to C}
\qquad
\infer[(\to\cdot)]{\Gamma_1, \Gamma_2 \to A \cdot B}{\Gamma_1 \to A & \Gamma_2 \to B}
$$

$$
\infer[(\BS\to)]{\Gamma_1, \Pi, A \BS B, \Gamma_2 \to C}
{\Pi \to A & \Gamma_1, B, \Gamma_2 \to C}
\qquad
\infer[(\to\BS)]{\Pi \to A \BS B}{A, \Pi \to B}
$$

$$
\infer[(\SL\to)]{\Gamma_1, B \SL A, \Pi, \Gamma_2 \to C}
{\Pi \to A & \Gamma_1, B, \Gamma_2 \to C}
\qquad
\infer[(\to\SL)]{\Pi \to  B \SL A}{\Pi, A \to B}
$$

$$
\infer[(\U\to)]{\Gamma_1, \U, \Gamma_2 \to C}{\Gamma_1, \Gamma_2 \to C}
\qquad
\infer[(\to\U)]{\to\U}{}
$$

$$
\infer[(\vee\to)]{\Gamma_1, A_1 \vee A_2, \Gamma_2 \to C}
{\Gamma_1, A_1, \Gamma_2 \to C & \Gamma_1, A_2, \Gamma_2 \to C}
\qquad
\infer[(\to\vee),\text{ where $i = 1$ or 2}]{\Gamma \to A_1 \vee A_2}{\Gamma \to A_i}
$$

$$
\infer[(\wedge\to),\text{ where $i = 1$ or 2}]{\Gamma_1, A_1 \wedge A_2, \Gamma_2 \to C}
{\Gamma_1, A_i, \Gamma_2 \to C}
\qquad
\infer[(\to\wedge)]{\Gamma \to A_1 \wedge A_2}{\Gamma \to A_1 & \Gamma \to A_2}
$$

$$
\infer[({!}\to)]{\Gamma_1, {!}^s A, \Gamma_2 \to C}{\Gamma_1, A, \Gamma_2 \to C}
\qquad
\infer[(\to{!}),\text{ where $s_j \succeq s$ for all $j$}]{{!}^{s_1} A_1, 
\dots, {!}^{s_n} A_n \to {!}^s B}
{{!}^{s_1} A_1, 
\ldots, {!}^{s_n} A_n \to B}
$$

$$
\infer[(\WEAK),\text{ where $s \in \Wc$}]{\Gamma_1, {!}^s A, \Gamma_2 \to C}
{\Gamma_1, \Gamma_2 \to C}
$$

$$
\infer[(\NCONTR_1)]{\Gamma_1, {!}^s A, \Delta, \Gamma_2 \to C}
{\Gamma_1, {!}^s A, \Delta, {!}^s A, \Gamma_2 \to C}
\quad\text{ \raisebox{1em}{and} } \quad
\infer[(\NCONTR_2),\text{ where $s \in \Cc$}]{\Gamma_1,  \Delta, {!}^s A, \Gamma_2 \to C}
{\Gamma_1, {!}^s A, \Delta, {!}^s A, \Gamma_2 \to C}
$$

$$
\infer[(\EXC_1)]{\Gamma_1, {!}^s A, \Delta, \Gamma_2 \to C}
{\Gamma_1, \Delta, {!}^s A, \Gamma_2 \to C}
\quad\text{ \raisebox{1em}{and} } \quad
\infer[(\EXC_2),\text{ where $s \in \Ec$}]{\Gamma_1,  \Delta, {!}^s A, \Gamma_2 \to C}
{\Gamma_1, {!}^s A, \Delta, \Gamma_2 \to C}
$$

$$
\infer[(\CUT)]{\Gamma_1, \Pi, \Gamma_2 \to C}{\Pi \to A & \Gamma_1, A, \Gamma_2 \to C}
$$



Due to the special status of the cut rule, we always explicitly state whether we're using
it in our derivations. Namely, we use the notation $\LambekA$ for the cut-free calculus and
$\LambekA + (\CUT)$ for the calculus with the cut rule.

It is sufficient to postulate $(\AX)$ only for variables, in the form $p_i \to p_i$. All other
instances of $A \to A$ are then derivable in a standard manner, without using $(\CUT)$.
For the subexponential case, derivability of ${!}^s A \to {!}^s A$ is due to the reflexivity 
of ${\preceq}$.

In Section~\ref{S:embed} we prove the cut elimination theorem for $\LambekA$
(Corollary~\ref{Cor:Lambekcutelim}), that is, $\LambekA + (\CUT)$ and
$\LambekA$ derive the same set of theorems. This
yields the subformula property, and therefore it becomes very easy to consider
fragments of the system by restricting
the language. If we take only rules that operate multiplicative connectives,
$\cdot$, $\BS$, and $\SL$, and rules that operate subexponentials, ${!}^s$ 
($s \in \Ic$), we obtain
the subexponential extension of the ``pure'' Lambek calculus, 
denoted by $\Lambek$.
If we also take the unit constant, $\U$, we get the calculus $\Lambek^\U$. 
Finally, removing rules for subexponentials yields, respectively, the Lambek
calculus allowing empty antecedents~\citep{Lambek61} and the Lambek calculus with the unit~\citep{Lambek69}.
All these calculi are conservative fragments of $\LambekA$.

Notice that the version of the Lambek calculus considered in this paper 
allows the antecedents
of sequents to be empty, while the original system by~\citet{Lambek58} doesn't.
This constraint, called {\em Lambek's restriction,} is motivated by linguistic
applications of the Lambek calculus. This restriction, however, appears to be
incompatible with (sub)exponential modalities~\citep{KanKuzSceAPAL,KanKuzSceLFCS}.

\section{Cyclic Linear Logic with Subexponentials}\label{S:Linlog}

In this section we define the second calculus considered in this paper, the extension of
{\em cyclic linear logic}~\citep{Yetter}
with subexponentials. For a subexponential signature $\Sigma = \langle \Ic, {\preceq}, \Wc, \Cc, \Ec \rangle$, this calculus is denoted by $\Linlog$.

We formulate $\Linlog$ in a language with {\em tight negations.} For a countable set of variables
$\Var = \{ p_1, p_2, \ldots \}$, we also consider their negations $\bar{p}_1, \bar{p}_2, \ldots$;
variables and their negations are called {\em atoms.}
Formulae of $\Linlog$ are built from atoms and constants $\U$ (multiplicative truth), $\bot$
(multiplicative falsity), $\top$ (additive truth), and $\Z$ (additive falsity) using four binary connectives:
$\Tensor$ (multiplicative conjunction), $\Par$ (multiplicative disjunction), $\AND$ (additive conjunction),
and $\OR$ (additive disjunction), and also two families of unary connectives, indexed by the subexponential
signature $\Sigma$: $!^{s}$ (universal subexponential) and $?^{s}$ (existential subexponential) for each $s \in \Ic$
(recall that $\Sigma = \langle \Ic, {\preceq}, \Wc, \Cc, \Ec \rangle$, and $\Ic$ is the set of all subexponential labels).

Negation for arbitrary formulae introduced externally by the following recursive definition ($A^\bot$ means ``not $A$''):
\begin{align*}
& p_i^\bot = \np_i && ({!}^s A)^\bot = {?}^s A^\bot\\
& \np_i^\bot = p_i && ({?}^s A)^\bot = {!}^s A^\bot\\
& (A \Tensor B)^\bot = B^\bot \Par A^\bot && \U^\bot = \bot\\
& (A \Par B)^\bot = B^\bot \Tensor A^\bot && \bot^\bot = \U\\
& (A \OR B)^\bot = A^\bot \AND B^\bot && \Z^\bot = \top \\
& (A \AND B)^\bot = A^\bot \OR B^\bot && \top^\bot = \Z
\end{align*}

Sequents of $\Linlog$ are of the form $\vdash \Gamma$ where $\Gamma$ in $\Linlog$ is a non-empty {\em cyclically ordered sequence:} 
sequents $\vdash \Gamma_1, \Gamma_2$ and $\vdash \Gamma_2, \Gamma_1$ are considered graphically equal, but other permutations of formulae
within $\vdash \Gamma$ are not allowed.

The axioms and rules of inference of $\Linlog$ are as follows:

$$
\infer[(\AX)]{\vdash A, A^\bot}{}
$$



$$
\infer[(\Tensor)]{\vdash \Gamma, A \Tensor B, \Delta}{\vdash \Gamma, A & \vdash B, \Delta}
\qquad
\infer[(\Par)]{\vdash A \Par B, \Gamma}{\vdash A, B, \Gamma}
$$

$$
\infer[(\AND)]{\vdash A_1 \AND A_2, \Gamma}{\vdash A_1, \Gamma & \vdash A_2, \Gamma}
\qquad
\infer[(\OR),\text{ where $i=1$ or 2}]{\vdash A_1 \OR A_2, \Gamma}{\vdash A_i, \Gamma}
$$

$$
\infer[(\U)]{\hspace*{.5em}\vdash \U\hspace*{.5em}}{}
\qquad
\infer[(\bot)]{\vdash \bot, \Gamma}{\vdash \Gamma}
\qquad
\infer[(\top)]{\vdash \top, \Gamma}{}
$$


$$
\infer[(!)\text{, where $s_j \succeq s$ for all $j$}]{\vdash {!}^{s} B, {?}^{s_1} A_1, \ldots, {?}^{s_n} A_n}{\vdash B, {?}^{s_1} A_1, \ldots, {?}^{s_n} A_n}
$$

$$
\infer[(?)]{\vdash {?}^s A, \Gamma}{\vdash A, \Gamma}
$$

$$
\infer[(\WEAK)\text{, where $s \in \Wc$}]{\vdash {?}^s A, \Gamma}{\vdash \Gamma}
$$

$$
\infer[(\NCONTR)\text{, where $s \in \Cc$}]{\vdash {?}^s A, \Gamma, \Delta}
{\vdash {?}^s A, \Gamma, {?}^s A, \Delta}
$$

$$
\infer[(\EXC)\text{, where $s \in \Ec$}]{\vdash {?}^s A, \Gamma, \Delta}
{\vdash \Gamma, {?}^s A, \Delta}
$$

$$
\infer[(\CUT)]{\vdash \Gamma, \Delta}{\vdash \Gamma, A^\bot & A, \Delta}
$$

Note that there is no rule for additive falsity, $\Z$. The only way to introduce
$\Z$ is by $(\top)$ or $(\AX)$, yielding $\vdash \top, \Gamma_1, \Z, \Gamma_2$
(if we use $(\AX)$, $\Gamma_1$ and $\Gamma_2$ are empty).


Also notice that we can freely apply cyclic transformations
to our sequents, yielding rules of the form
$$
\infer[(\Par)]
{\vdash \Gamma_1, A \Par B, \Gamma_2}{\vdash \Gamma_1, A, B, \Gamma_2}
\qquad
\infer[(\Tensor_1)]
{\vdash \Gamma_1, A \Tensor B, \Delta, \Gamma_2}{\vdash \Gamma_1, A, \Gamma_2 & \vdash B, \Delta}
\qquad
\infer[(\Tensor_2)]
{\vdash \Gamma_1, \Delta, A \Tensor B, \Gamma_2}{\vdash \Delta, A & \vdash \Gamma_1, B, \Gamma_2}
$$
and so on. Due to our conventions, these rules are actually {\em graphically equal} to the 
official rules of $\Linlog$ presented above. Sometimes, however, these transformed
forms of the rules are more convenient---for example, if we want a specific designated
formula to be the rightmost one (see proof of Theorem~\ref{Th:embed}).

As in $\LambekA$, in $\Linlog$ it is sufficient to postulate $(\AX)$ only for variables, as
$\vdash p_i, \np_i$.

As for the Lambek calculus, we use the notation $\Linlog$ for the cut-free system,
and $\Linlog + (\CUT)$ for the system with cut. In Section~\ref{S:cutelim} we establish
cut elimination, that yields the subformula property. If we remove all additives connectives
and rules for them, leaving only $\U$, $\bot$, $\Tensor$, $\Par$, and the subexponentials,
we get the {\em multiplicative} fragment of cyclic linear logic with subexponentials,
denoted by $\mathrm{SMCLL}_\Sigma$.

\section{Cut Elimination in $\Linlog$}\label{S:cutelim}



\begin{theorem}\label{Th:cutelim}
A sequent is derivable in $\Linlog + (\CUT)$ if and only if it is derivable in
$\Linlog$.
\end{theorem}

The cut elimination strategy we use here goes back to~\citet{Gentzen}, and
was applied for linear logic by~\citet{Girard}. We follow the outline of
the proof presented in~\citep[Appendix~A]{LMSS}, making necessary
modifications for the cases where exchange rules are not available.

Since eliminating the cut rule by straightforward induction encounters problems when
it comes across the contraction rule, we consider 
the cut rule together with a more general rule called {\em mix,} which is a combination
of cut and contraction. The two rules can now be eliminated by joint induction
(which is impossible for the original cut rule alone).

Another possible cut elimination strategy for $\Linlog$ is ``deep cut
elimination'' of~\citet{dePaiva}. 
This strategy is applied by~\citet{KanKuzSceFCT} to establish cut elimination 
in a system closely related to $\Lambek$, but with bracket modalities that
introduce controlled non-associativity, which makes it hard to formulate the mix rule.
In this paper we follow the more traditional approach.

Since mix needs contraction, it is included only for formulae of the form $?^{s} A$ with $s \in \Cc$.
Thus, unlike the classic Gentzen's situation, $(\CUT)$ is not always a particular case of $(\MIX)$, and in our
proof we eliminate both cut and mix by joint induction.

If $s \in \Cc \cap \Ec$ ({\em i.\,e.,} $?^{s}$ also allows exchange---in particular, this is the
case for the ``full-power'' exponential connective of linear logic), the mix rule can be formulated
exactly as in the commutative case:
$$
\infer[(\MIX)]
{\vdash \Gamma, \Delta}
{\vdash \Gamma, !^{s} A^\bot & \vdash ?^{s}A, \ldots, ?^{s}A, \Delta}
$$

For $s \in \Cc - \Ec$, however, the formulation of mix is more sophisticated, since we are not allowed
to gather all instances of $?^{s}A$ in one area of the sequent:
$$
\infer[(\MIX)]
{\vdash \Gamma, \Delta_1, \Delta_2, \ldots, \Delta_k}
{\vdash \Gamma, !^{s} A^\bot & \vdash ?^{s} A, \Delta_1, ?^{s} A, \Delta_2, \ldots, {?}^s A, \Delta_k}
$$
In this rule, one instance of $?^{s} A$ is replaced with $\Gamma$ (due to cyclicity we can suppose that
it is the leftmost occurrence), and several (maybe zero) other occurrences of $?^{s} A$ are removed from
the sequent.

Being equivalent to a consequent application of several $(\NCONTR)$'s and $(\CUT)$, the mix rule is clearly
admissible in $\Linlog + (\CUT)$. 

As in the commutative case, cut elimination crucially depends on the fact that
the {$\preceq$} relation is transitive and that the sets $\Wc$, $\Cc$, and $\Ec$ are upwardly closed w.r.t.\ {$\preceq$}.
These parts of the definition of the substructural signature $\Sigma$ come into play when
we propagate $(\CUT)$ or $(\MIX)$ through the $({!})$ rule that yields
$\vdash {?}^{s_1} C_1, \ldots, {?}^{s_n} C_n, {!}^s A^\bot$. In this situation,
the formula ${?}^s A$ get replaced by a sequence ${?}^{s_1} C_1, \ldots, {?}^{s_n} C_n$, and we need the same
structural rules to be valid for ${?}^{s_i} C_i$, as for ${?}^s A$. This is guaranteed by the fact that
$s_i \succeq s$ (a prerequisite of the $({!})$ rule) and the closure properties of $\Sigma$.

In the non-commutative situation, however, there is another issue one should be cautious about. 
For cut elimination, it is important that the contraction rule is non-local, {\em i.e.,} the formulae
being contracted can come from distant places of the sequent, with other formulae ($\Gamma$) between
them. Accordingly, our formulation of $(\MIX)$ for subexponentials that allow contraction, but not
exchange, is also non-local, with $\Delta_i$ between the active formulae. In Section~\ref{S:CutVsContr}
we show that for the local version of contraction, that allows contracting only neighbour formulae,
cut elimination doesn't hold.

\begin{proof}[Proof of Theorem~\ref{Th:cutelim}.]
As usual, it is sufficient to eliminate one cut or mix, {\em i.e.,} to show the following
two statements:
\begin{itemize}
\item if both $\vdash \Gamma, A^\bot$ and $\vdash A, \Delta$ are derivable in $\Linlog$,
then so is $\vdash \Gamma, \Delta$;
\item if $s \in \Cc$ and both $\vdash \Gamma, {!}^s A^\bot$ and
$\vdash {?}^s A, \Delta_1, {?}^s A, \Delta_2, \ldots, {?}^s A, \Delta_k$ are derivable
in $\Linlog$, then so is $\vdash \Gamma, \Delta_1, \ldots, \Delta_k$.
\end{itemize}

We prove both statements by joint nested induction. The outer induction
parameter is $\kappa$, the total number of connectives in the formula being cut
(for $(\MIX)$, the external ${?}^s$ also counts). The inner induction parameter 
is $\delta$, the sum of the heights of the cut-free derivations of two premises,
$\vdash \Gamma, A^\bot$ and $\vdash A, \Delta$ for $(\CUT)$ and
$\vdash \Gamma, {!}^s A^\bot$ and
$\vdash {?}^s A, \Delta_1, {?}^s A, \Delta_2, \ldots, {?}^s A, \Delta_k$ for $(\MIX)$.
At each step, either $\kappa$ decreases, or $\delta$ decreases while $\kappa$ remains 
the same.

The cut (mix) elimination procedure is usually very lengthy and tedious, since it
requires considering a great number of cases and subcases of which rules are
the last rules applied in the (cut-free) derivations of the premises of $(\CUT)$ or
$(\MIX)$. Here we try to make it as short as possible by merging similar cases. 

\subsection*{Cut Elimination}
The cut elimination procedure is a rather standard, straightforward induction. When
we come across the $(\NCONTR)$ rule, $(\CUT)$ becomes $(\MIX)$, and we jump to the
second, more interesting part of the proof.

The last rule applied in the derivation of $\vdash \Gamma, A^\bot$ (or, symmetrically,
$\vdash A, \Delta$) is called {\em principal} either if it is an application
of the $({!})$ rule or if it introduces the rightmost $A^\bot$ (symmetrically,
the leftmost $A$) formula. Otherwise it is called {\em non-principal.}

{\it Case 1.} One of the cut premises is an axiom of the form $(\AX)$.
Then the goal sequent coincides with the other premise, and cut disappears.

{\it Case 2.} The last rule in the derivation either of $\vdash \Gamma, A^\bot$ or of
$\vdash A, \Delta$ is non-principal.

Since $A^{\bot\bot} = A$, the cut (but not mix) rule is {\em symmetric.} Therefore,
we don't have to consider both $\vdash \Gamma, A^\bot$ and
$\vdash A, \Delta$; handling only $\vdash \Gamma, A^\bot$ is sufficient.

Let us call $(\Par)$, $(\OR)$, $(\bot)$, $(?)$, $(\WEAK)$, $(\NCONTR)$, and
$(\EXC)$ {\em easy} rules. An easy rule doesn't branch the derivation, it
only transforms the sequent, and, in the non-principal case, keeps the formula
being cut intact. If $\vdash \Gamma, A^\bot$ is derived using
an easy rule, the cut application has the following form
(``ER'' stands for ``easy rule''):
$$
\infer[(\CUT)]{\vdash \Gamma, \Delta}{\infer[(\mathrm{ER})]{\vdash \Gamma, A^\bot}
{\vdash \widetilde{\Gamma}, A^\bot} & 
\vdash A, \Delta}
$$
and the cut is propagated:
$$
\infer[(\mathrm{ER})]{\vdash \Gamma, \Delta}
{\infer[(\CUT)]{\vdash \widetilde{\Gamma}, \Delta}{\vdash \widetilde{\Gamma}, A^\bot &
\vdash A, \Delta}}
$$
The easy rule is still valid in a different context. The new cut has the same $\kappa$
and a smaller $\delta$ parameter, and gets eliminated by induction.

The other non-principal cases, $(\Tensor)$, $(\AND)$, and $(\top)$, are handled
as follows:
$$
\infer[(\CUT)]{\vdash \Gamma_1, \Phi, E \Tensor F, \Gamma_2, \Delta}
{\infer[(\Tensor)]{\vdash \Gamma_1, \Phi, E \Tensor F, \Gamma_2, A^\bot}
{\vdash \Phi, E & \vdash \Gamma_1, F, \Gamma_2, A^\bot} & 
\vdash A, \Delta}
\TRANSFORM
\infer[(\Tensor)]{\vdash \Gamma_1, \Phi, E \Tensor F, \Gamma_2, \Delta}
{\vdash \Phi, E & \infer[(\CUT)]{\vdash \Gamma_1, F, \Gamma_2, \Delta}
{\vdash \Gamma_1, F, \Gamma_2, A^\bot & \vdash A, \Delta}}
$$
(The case when $A^\bot$ goes to the branch with $F$ is symmetric.)
$$
\infer[(\CUT)]{\vdash \Gamma_1, E_1 \AND E_2, \Gamma_2, \Delta}
{\infer[(\AND)]{\vdash \Gamma_1, E_1 \AND E_2, \Gamma_2, A^\bot} 
{\vdash \Gamma_1, E_1, \Gamma_2, A^\bot & \vdash \Gamma_1, E_1 \AND E_2, \Gamma_2, A^\bot}
& 
\vdash A, \Delta}
$$
becomes
$$
\infer[(\AND)]{\vdash \Gamma_1, E_1 \AND E_2, \Gamma_2, \Delta}
{\infer[(\CUT)]{\vdash \Gamma_1, E_1, \Gamma_2, \Delta}
{\vdash \Gamma_1, E_1, \Gamma_2, A^\bot & \vdash A, \Delta} &
\infer[(\CUT)]{\vdash \Gamma_1, E_2, \Gamma_2, \Delta}
{\vdash \Gamma_1, E_2, \Gamma_2, A^\bot & \vdash A, \Delta}
}
$$

$$
\infer[(\CUT)]{\vdash \Gamma_1, \top, \Gamma_2, \Delta}
{\infer[(\top)]{\vdash \Gamma_1, \top, \Gamma_2, A^\bot}{} & \vdash A, \Delta}
\TRANSFORM
\infer[(\top)]{\vdash \Gamma_1, \top, \Gamma_2, \Delta}{}
$$
For $(\Tensor)$ and $(\AND)$, the $\delta$ parameter decreases with the
same $\kappa$. For $(\top)$, cut disappears.

Applications of $(\U)$ and $({!})$ cannot be non-principal.

{\it Case 3.} The last rules in both derivations are principal, and the main 
connective of $A$ is not a subexponential. Consider the possible pairs of principal
rules; due to symmetry of cut, the order in these pairs doesn't matter.

{\it Subcase 3.1.} $(\Tensor)$ and $(\Par)$

$$
\infer[(\CUT)]{\vdash\Gamma_1, \Gamma_2, \Delta}
{\infer[(\Tensor)]{\vdash \Gamma_1, \Gamma_2, F^\bot \Tensor E^\bot}
{\vdash \Gamma_2, F^\bot & \vdash \Gamma_1, E^\bot} 
& \infer[(\Par)]{\vdash E \Par F, \Delta}{\vdash E, F, \Delta}}
\TRANSFORM
\infer[(\CUT)]{\vdash\Gamma_1, \Gamma_2, \Delta}
{\Gamma_2, F^\bot & \infer[(\CUT)]{\Gamma_1, F, \Delta}
{\Gamma_1, E^\bot & E, F, \Delta}}
$$
The $\kappa$ parameter for both new cuts is less than $\kappa$ of the
original cut, thus we cant proceed by induction.

{\it Subcase 3.2.} $(\AND)$ and $(\OR)$

$$
\infer[(\CUT)]{\vdash \Gamma, \Delta}
{\infer[(\AND)]{\vdash \Gamma, E_1^\bot \AND E_2^\bot}
{\vdash \Gamma, E_1^\bot & \vdash \Gamma, E_2^\bot} & 
\infer{\vdash E_1 \OR E_2, \Delta}{\vdash E_i, \Delta}}
\TRANSFORM
\infer[(\CUT)]{\vdash \Gamma, \Delta}
{\vdash \Gamma, E_i^\bot & \vdash E_i, \Delta}
$$
Again, $\kappa$ gets decreased.

{\it Subcase 3.3.} $(\U)$ and $(\bot)$

$$
\infer[(\CUT)]{\vdash\Delta}
{\infer[(\U)]{\vdash \U}{} & \infer[(\bot)]{\vdash \bot, \Delta}{\vdash \Delta}}
$$
Cut disappears, since its goal coincides with the premise of $(\bot)$, which is
already derived.

In the principal case, we don't need to consider the $(\top)$ rule, since it has no principal
counterpart that introduces $\top^\bot = \Z$.

{\it Case 4.} Both last rules are principal, $A = {?}^s B$, and
$A^\bot = {!}^s B^\bot$. The left premise, $\vdash \Gamma, {!}^s B^\bot$, is
derived using $(!)$ introducing ${!}^s B^\bot$. Therefore,
$\Gamma = {?}^{s_1} C_1, \ldots, {?}^{s_n} C_n$, where $s_i \succeq s$ for all $i$.
Consider the possible cases for the last rule in the derivation of the
other premise, $\vdash {?}^s A, \Delta$.

{\it Subcase 4.1.} The last rule is $({?})$:
$$
\infer[(\CUT)]{\vdash \Gamma, \Delta}
{\infer[(!)]{\vdash \Gamma, {!}^s B^\bot}{\vdash \Gamma, B^\bot} & 
\infer[(?)]{\vdash {?}^s B, \Delta}{\vdash B, \Delta}}
\TRANSFORM
\infer[(\CUT)]{\vdash \Gamma, \Delta}
{{\vdash \Gamma, B^\bot} & 
{\vdash B, \Delta}}
$$
The $\kappa$ parameter gets decreased.


{\it Subcase 4.2.} The last rule is $({!})$:
$$
\infer[(\CUT)]
{\vdash {?}^{s_1} C_1, \ldots, {?}^{s_n} C_n, {?}^{q_1} D_1, \ldots,
{?}^{q_i} D_i, {!}^q E, {!}^{q_{i+1}} D_{i+1}, \ldots, {?}^{q_m} D_m}
{\vdash {?}^{s_1} C_1, \ldots, {?}^{s_n} C_n, {!}^s B^\bot & 
\infer[(!)]{\vdash {?}^s B, {?}^{q_1} D_1, \ldots,
{?}^{q_i} D_i, {!}^q E, {!}^{q_{i+1}} D_{i+1}, \ldots, {?}^{q_m} D_m}
{\vdash {?}^s B, {?}^{q_1} D_1, \ldots,
{?}^{q_i} D_i,  E, {!}^{q_{i+1}} D_{i+1}, \ldots, {?}^{q_m} D_m}}
$$
becomes
$$
\infer[(!)]
{\vdash {?}^{s_1} C_1, \ldots, {?}^{s_n} C_n, {?}^{q_1} D_1, \ldots,
{?}^{q_i} D_i, {!}^q E, {!}^{q_{i+1}} D_{i+1}, \ldots, {?}^{q_m} D_m}
{\infer[(\CUT)]{\vdash {?}^{s_1} C_1, \ldots, {?}^{s_n} C_n, {?}^{q_1} D_1, \ldots,
{?}^{q_i} D_i, E, {!}^{q_{i+1}} D_{i+1}, \ldots, {?}^{q_m} D_m}
{\vdash {?}^{s_1} C_1, \ldots, {?}^{s_n} C_n, {!}^s B^\bot & 
\vdash {?}^s B, {?}^{q_1} D_1, \ldots,
{?}^{q_i} D_i,  E, {!}^{q_{i+1}} D_{i+1}, \ldots, {?}^{q_m} D_m}
}
$$
where the new application of $(!)$ is legal due to transitivity of
$\preceq$: $s_i \succeq s \succeq q$.
The $\kappa$ parameter is the same, $\delta$ decreases.

{\it Subcase 4.3.} The last rule is $(\WEAK)$. In this case, since 
$s_i \succeq s$ and $s \in \Wc$, then also $s_i \in \Wc$, and $\Gamma = 
{?}^{s_1} C_1, \ldots, {?}^{s_n} C_n$ can be added to $\Delta$ using 
the weakening rule $n$ times. Cut disappears.

{\it Subcase 4.4.} The last rule is $(\NCONTR)$. In this case cut is
replaced mix with the same $\kappa$ and a smaller $\delta$:
$$
\infer[(\CUT)]{\vdash \Gamma, \Delta_1, \Delta_2}
{\vdash \Gamma, {!}^s B^\bot & 
\infer[(\NCONTR)]{\vdash {?}^s B, \Delta_1, \Delta_2}
{\vdash {?}^s B, \Delta_1, {?}^s B, \Delta_2}}
\TRANSFORM
\infer[(\MIX)]{\vdash \Gamma, \Delta_1, \Delta_2}
{\vdash \Gamma, {!}^s B^\bot & 
{\vdash {?}^s B, \Delta_1, {?}^s B, \Delta_2}}
$$


{\it Subcase 4.5.} The last rule is $(\EXC)$. Similarly to Subcase~4.3, $s_i \in \Ec$,
and we can apply the exchange rule for $\Gamma$ as a whole. This means that $(\CUT)$ can
be interchanged with $(\EXC)$, decreasing $\delta$ with the same $\kappa$.

\subsection*{Mix Elimination}

For the left premise, the definition of principal rule is the same as for $(\CUT)$.
For the right one, a rule is principal if it is $({!})$ or operates with one of
the ${?}^s A$ formulae used in $(\MIX)$.
Eliminating mix is easier, since now  principal rules
could be only rules for subexponentials, and thus we have to consider
a smaller number of cases. Moreover, we can assume that $k \geq 2$,
since mix with $k = 1$ is actually cut.

{\em Case 1.} One of the mix premises is an axiom of the form $(\AX)$. Then,
as for $(\CUT)$, the goal coincides with the other premise.

{\em Case 2.} The last rule in the derivation of the left premise,
$\vdash \Gamma, {!}^s A^\bot$, is non-principal. In this case we proceed exactly
as in the non-principal case for $(\CUT)$: the mix rule gets propagated to
the left, and $\delta$ decreases with the same $\kappa$.

{\em Case 3.} The last rule in the left derivation is principal and the last rule in 
the right one is non-principal. In this case the rule on the left is $(!)$, 
introducing ${!}^s A^\bot$.
The interesting situation here is the $(\Tensor)$ rule yielding the right
premise, ${?}^s A, \Delta_1, {?}^s A, \Delta_2, \ldots, {?}^s A, \Delta_k$.
The derivation branches, and there are two possibilites: either all instances
of ${?}^s A$ involved in $(\MIX)$ go to one branch, or they split between branches.

If they don't split, the transformation is again the same as for cut elimination:
$$
\infer[(\MIX)]{\vdash \Gamma, \Delta_1, \ldots, \Delta'_i, \Phi, E \Tensor F, \Delta''_i,
\ldots, \Delta_k}{\vdash \Gamma, {!}^s A^\bot & 
\infer[(\Tensor)]{\vdash {?}^s A, \Delta_1, \ldots, {?}^s A, \Delta'_i, \Phi, E \Tensor F, \Delta''_i,
\ldots, {?}^s A, \Delta_k}
{\vdash \Phi, E & \vdash {?}^s A, \Delta_1, \ldots, {?}^s A, \Delta'_i, F, \Delta''_i,
\ldots, {?}^s A, \Delta_k}}
$$
becomes
$$
\infer[(\Tensor)]{\vdash \Gamma, \Delta_1, \ldots, \Delta'_i, \Phi, E \Tensor F, \Delta''_i,
\ldots, \Delta_k}
{\vdash \Phi, E & \infer[(\MIX)]{\vdash \Gamma, \Delta_1, \ldots, \Delta'_i, F, \Delta''_i,
\ldots, \Delta_k}{\vdash \Gamma, {!}^s A^\bot &
 \vdash {?}^s A, \Delta_1, \ldots, {?}^s A, \Delta'_i, F, \Delta''_i,
\ldots, {?}^s A, \Delta_k}}
$$

The situation with splitting is more involved. In this case we recall that
$\vdash \Gamma, {!}^s A^\bot$ is obtained by application of ${!}$, therefore
$\Gamma = {?}^{s_1} C_1, \ldots, {?}^{s_n} C_n$, where $s_i \succeq s$ for all $i$.
Hence, $s_i \in \Cc$, and we can apply the non-local contraction rule for
formulae in $\Gamma$. Then we first apply $(\MIX)$ to both premises of $(\Tensor)$,
apply $(\Tensor)$ and arrive at a sequent with two occurrences of $\Gamma$, that
are merged by applying the $(\NCONTR)$  rule $n$ times. An example of such transformation
is presented below (the case where the leftmost ${?}^s A$ goes with $E$ instead of $F$
is symmetric):
$$
\infer[(\MIX)]{\vdash \Gamma, \Delta_1, \ldots, \Delta''_j, \Delta'_j, \ldots,
\Delta'_i, E \Tensor F, \Delta''_i, \ldots, \Delta_k}
{\vdash \Gamma, {!}^s A^\bot & 
\infer[(\Tensor)]{\vdash {?}^s A, \Delta_1, \ldots, {?}^s A, \Delta''_j, \Delta'_j, {?}^s A, \ldots,
{?}^s A, \Delta'_i, E \Tensor F, \Delta''_i, \ldots, {?}^s A, \Delta_k}
{\vdash \Delta'_j, {?}^s A, \ldots, {?}^s A, \Delta'_i, E & 
\vdash {?}^s A, \Delta_1, \ldots, {?}^s A, \Delta''_j, F, \Delta''_i, \ldots, {?}^s A, 
\Delta_k}}
$$
becomes
$$
\footnotesize
\infer[(\NCONTR)\text{ $n$ times}]{\vdash \Gamma, \Delta_1, \ldots, \Delta''_j, \Delta'_j,
\ldots,
\Delta'_i, E \Tensor F, \Delta''_i, \ldots, \Delta_k}
{\infer[(\Tensor)]{\vdash \Gamma, \Delta_1, \ldots, \Delta''_j, \Delta'_j, \Gamma, \ldots,
\Delta'_i, E \Tensor F, \Delta''_i, \ldots, \Delta_k}
{\infer[(\MIX)]{\vdash \Delta'_j, \Gamma, \ldots,
\Delta'_i, E}{\vdash \Gamma, {!}^s A^\bot & 
\vdash \Delta'_j, {?}^s A, \ldots, {?}^s A, \Delta'_i, E} &  
\infer[(\MIX)]{\vdash \Gamma, \Delta_1, \ldots, \Delta''_j, F, \Delta''_i, \ldots, \Delta_k}
{\vdash \Gamma, {!}^s A^\bot & \vdash {?}^s A, \Delta_1, \ldots, {?}^s A, \Delta''_j, F, \Delta''_i, \ldots, {?}^s A, 
\Delta_k}}}
$$
Both new applications of $(\MIX)$ have a smaller $\delta$ with the same $\kappa$, and
we proceed by induction.

All other non-principal cases (easy rules, $(\top)$, and $(\AND)$) are handled exactly as 
in the non-principal case for $(\CUT)$, only the notation becomes a bit longer.

{\em Case 4.} The last rule in both derivations is principal.
Then, again, the left premise is $({!})$ introducing
${!}^s A^\bot$, whence $\Gamma = 
{?}^{s_1} C_1, \ldots, {?}^{s_n} C_n$, and we consider subcases on which rule
is used on the right.

{\em Subcase 4.1.} The last rule is $({?})$. If this rule introduces the leftmost
instance of ${?}^s A$, the transformation is as follows (recall that
$k \geq 2$):
$$ 
\infer[(\MIX)]{\vdash \Gamma, \Delta_1, \Delta_2, \ldots, \Delta_k}
{\infer[(!)]{\vdash \Gamma, {!}^s A^\bot}{\vdash \Gamma, A^\bot} & 
\infer[(?)]{\vdash {?}^s A, \Delta_1, {?}^s A, \Delta_2, \ldots, {?}^s A, \Delta_k}
{\vdash A, \Delta_1, {?}^s A, \Delta_2, \ldots, {?}^s A, \Delta_k}}
$$
becomes
$$
\infer[(\NCONTR)\text{ several times}]
{\vdash \Gamma, \Delta_1, \Delta_2, \ldots, \Delta_k}
{\infer[(\CUT)]{\vdash \Gamma, \Delta_1, \Gamma, \Delta_2, \ldots, \Delta_k}
{\vdash \Gamma, A^\bot & 
\infer[(\MIX)]{\vdash A, \Delta_1, \Gamma, \Delta_2, \ldots, \Delta_k}
{\vdash \Gamma, {!}^s A^\bot & 
\vdash A, \Delta_1, {?}^s A, \Delta_2, \ldots, {?}^s A, \Delta_k}}}
$$
For $(\MIX)$, $\kappa$ is the same and $\delta$ gets decreased.
For $(\CUT)$, $\kappa$ gets decreased ($A$ is simpler than ${?}^s A$), and
we don't care for $\delta$ (which is uncontrolled). Thus, both cut and mix
are eliminable by induction. Finally, $s_i \in \Cc$ (since $s_i \succeq s$), whence
$(\NCONTR)$ can be applied to formulae from $\Gamma$.

If the $(?)$ rule introduces another instance of ${?}^{s} A$, the translation
is the same, but the second $\Gamma$ could appear not after $\Delta_1$, but after
some other $\Delta_i$.

{\em Subcase 4.2.} The last rule is $({!})$. The same as Subcase~4.2 of
cut elimination.

{\em Subcase 4.3.} The last rule is $(\NCONTR)$. Our mix rule was specifically
designed to subsume $(\NCONTR)$:
$$
\infer[(\MIX)]{\vdash \Gamma, \Delta_1, \ldots, \Delta_k}
{\vdash \Gamma, {!}^s A^\bot & 
\infer[(\NCONTR)]{\vdash {?}^s A, \Delta_1, {?}^s A, \Delta_2, \ldots,
{?}^s A, \Delta_i', \Delta_i'', \ldots, {?}^s A, \Delta_k}
{\vdash {?}^s A, \Delta_1, {?}^s A, \Delta_2, \ldots,
{?}^s A, \Delta_i', {?}^s A, \Delta_i'', \ldots, {?}^s A, \Delta_k}}
$$ transforms into $$
\infer[(\MIX)]{\vdash \Gamma, \Delta_1, \ldots, \Delta_k}
{\vdash \Gamma, {!}^s A^\bot & 
{\vdash {?}^s A, \Delta_1, {?}^s A, \Delta_2, \ldots,
{?}^s A, \Delta_i', {?}^s A, \Delta_i'', \ldots, {?}^s A, \Delta_k}}
$$
The $\delta$ parameter gets reduced with the same $\kappa$.

{\em Subcase 4.4.} The last rule is $(\EXC)$.
If this rule didn't move the leftmost instance of ${?}^s A$, then it gets
subsumed by $(\MIX)$ exactly as $(\NCONTR)$ in the previous subcase.
If the leftmost instance of ${?}^s A$ gets moved, then we recall that
$\Gamma = {?}^{s_1} C_1, \ldots, {?}^{s_n} C_n$, and $s_i \in \Ec$ for
all $i$ by the definition of subexponential signature, 
since $s_i \succeq s$ and $s \in \Ec$. This means we can apply the exchange
rule for $\Gamma$ as a whole. In this case we first apply $(\MIX)$ (with
the same $\kappa$ and a smaller $\delta$) to the sequent before $(\EXC)$ and
then move $\Gamma$ to the correct place by several applications of $(\EXC)$.

{\em Subcase 4.5.} The last rule is $(\WEAK)$. Again, if it introduced an
instance of ${?}^s A$ different from the leftmost one, it is subsumed by
$(\MIX)$. If the leftmost instance gets weakened, then we apply mix to
the second ${?}^s A$ (recall that $k \geq 2$, so we have another instance),
and then exchange $\Gamma$ with $\Delta_1$. This is allowed, since,
by our definitions, $\Wc \cap \Cc \subseteq \Ec$, and $s_i \succeq s \in \Wc \cap \Cc$
($s$ is in $\Wc$, since we used $(\WEAK)$ and in $\Cc$, since we used $(\MIX)$).
Again, $\kappa$ is the same and $\delta$ gets reduced.

\end{proof}

\section{Embedding of $\LambekA$ into $\Linlog$ and Cut Elimination in
$\LambekA$}\label{S:embed}

In this section we define an extension to subexponentials of the standard embedding of Lambek formulae into cyclic linear logic. Lambek formula
$A$ translates into linear logic formula $\widehat{A}$.
\begin{align*}
& \widehat{p}_i = p_i && \widehat{\U} = \U\\
& \widehat{A \cdot B} = \widehat{A} \Tensor \widehat{B} && \widehat{{!}^s A} = {!}^s \widehat{A}\\
& \widehat{A \BS B} = \widehat{A}^\bot \Par \widehat{B} && \widehat{A \wedge B} = \widehat{A} \AND \widehat{B}\\
&  \widehat{B \SL A} = \widehat{B} \Par \widehat{A}^\bot && \widehat{A \vee B} = \widehat{A} \OR \widehat{B}
\end{align*}

For convenience, we also recall the definition
of negation in $\Linlog$ and present the negative translations (negations of
translations) of Lambek formulae:
\begin{align*}
& \widehat{p}_i^\bot = \np_i && \widehat{\U}^\bot = \bot\\
& (\widehat{A \cdot B})^\bot = \widehat{B}^\bot \Par \widehat{A}^\bot && (\widehat{{!}^s A})^\bot = {?}^s \widehat{A}^\bot\\
& (\widehat{A \BS B})^\bot = \widehat{B}^\bot \Tensor \widehat{A} && (\widehat{A \wedge B})^\bot = \widehat{A}^\bot \OR \widehat{B}^\bot\\
&  (\widehat{B \SL A})^\bot = \widehat{A} \Tensor \widehat{B}^\bot && (\widehat{A \vee B})^\bot = \widehat{A}^\bot \AND \widehat{B}^\bot
\end{align*}

For $\Pi = A_1, \ldots, A_k$ let $\widehat{\Pi}^\bot$ be
$\widehat{A}_k^\bot, \ldots, \widehat{A}_1^\bot$ (for left-hand sides of
Lambek sequents, we need only the negative translation).

\begin{theorem}\label{Th:embed}
The following statements are equivalent:
\begin{enumerate}
\item the sequent $\Pi \to B$ is derivable in $\LambekA$;
\item the sequent $\Pi \to B$ is derivable in $\LambekA+(\CUT)$;
\item the sequent $\vdash \widehat{\Pi}^\bot, \widehat{B}$ is derivable in $\Linlog + (\CUT)$;
\item the sequent $\vdash \widehat{\Pi}^\bot, \widehat{B}$ is derivable in $\Linlog$.
\end{enumerate}
\end{theorem}

This theorem yields both cut elimination for $\LambekA$ and embedding of $\LambekA$ into $\Linlog$.

\begin{corollary}\label{Cor:Lambekcutelim}
A sequent is derivable in $\LambekA + (\CUT)$ if and only if it is derivable in
$\LambekA$.
\end{corollary}

\begin{corollary}\label{Cor:embed}
The sequent $\Pi \to B$ is derivable in $\LambekA$ if and only if the sequent
$\vdash \widehat{\Pi}^\bot, \widehat{B}$ is derivable in $\Linlog$.
\end{corollary}

We prove Theorem~\ref{Th:embed} by establishing round-robin implications:
\mbox{$1 \Rightarrow 2 \Rightarrow 3 \Rightarrow 4 \Rightarrow 1$.}
The last implication, $4 \Rightarrow 1$, is a bit surprising, since the Lambek calculus
is in a sense ``intuitionistic,'' and CLL is ``classic'' (cf.~\citet{Chaudhuri}).
However, it becomes possible due to the restricted language used in
the Lambek calculus: it includes neither multiplicative disjunction ($\Par$),
nor negation, 
nor existential subexponentials (${?}^s$), 
nor additive constants ($\Z$ and $\top$).

In the commutative case, as shown by~\citet{Schellinx}, these are
exactly the restrictions under which intuitionistic linear logic is a
conservative fragment of classical linear logic. In our non-commutative case,
the situation is the same: $\LambekA$ in its restricted language gets conservatively
embedded into $\Linlog$, but 
extending the language and including some of the forbidden connectives leads to
failure of the conservativity claim.

Multiplicative disjunction and negation allow encoding
{\em tertium non datur,} $A \Par A^\bot$, which is intuitionistically invalid.

In the implication-only language, there is still a principle that is valid
classically, but not intuitionistically, called {\em Peirce's law}~\citep{Peirce}:
 $((X \Rightarrow Y) 
\Rightarrow X) \Rightarrow X$. Encoding Peirce's law in substructural logic
requires explicitly allowing contraction for the rightmost $X$ and weakening for $Y$,
like this: $(x \BS {?}^w y) \BS x \to {?}^c x$, where $w \in \Wc$ and $c \in \Cc$.
This would be a counter-example for the $4 \Rightarrow 1$ implication;
fortunately, formulae of the form ${?}^s A$ are outside the language of
$\LambekA$. The translation of this substructural form of Peirce's law into
cyclic linear logic, $\vdash \bar{x} \Tensor (\bar{x} \Par {?}^w y), {?}^c x$,
is derivable in $\Linlog$ with an appropriate substructural signature $\Sigma$:
$$
\infer[(\NCONTR)]{\vdash \bar{x} \Tensor (\bar{x} \Par {?}^w y), {?}^c x}
{\infer[(\Tensor)]{\vdash {?}^c x, \bar{x} \Tensor (\bar{x} \Par {?}^w y), {?}^c x}
{\infer[(?)]{\vdash {?}^c x, \bar{x}}{\infer[(\AX)]{\vdash x, \bar{x}}{}} & 
\infer[(\Par)]{\vdash \bar{x} \Par {?}^w y, {?}^c x}
{\infer[(\WEAK)]{\vdash \bar{x}, {?}^w y, {?}^c x}
{\infer[(?)]{\vdash \bar{x}, {?}^c x}{\infer[(\AX)]{\vdash \bar{x},x}{}}}}}}
$$

Finally, if one extends the Lambek calculus
with the $\Z$ constant governed by the following left rule
$$
\infer[(\Z\to)]{\Gamma_1, \Z, \Gamma_2 \to C}{}
$$
and no right rule~\citep{LambekZero}, the $4 \Rightarrow 1$ implication (where
$\widehat{\Z} = \Z$) also becomes false. This is established by a non-commutative
version of the counter-example by~\citet{Schellinx}:
$$
(r \SL (\Z \BS q)) \SL p, (s \SL p) \BS \Z \to r.
$$
Since the Lambek calculus with $\Z$ still has the cut elimination property
(as we don't need subexponentials and additives, one can prove it by 
simple induction, as in~\citep{Lambek58}), one can perform exhaustive
proof search and find out that this sequent is not derivable.
On the other hand, its translation into cyclic linear logic,
$\vdash \top \Tensor (s \Par \np), p \Tensor (\top \Par q) \Tensor \bar{r}, r$,
is derivable in $\Linlog$:
$$
\infer[(\Tensor)]
{\vdash \top \Tensor (s \Par \np), p \Tensor ((\top \Par q) \Tensor \bar{r}), r}
{\infer[(\top)]{\vdash \top}{} &
\infer[(\Par)]{\vdash s \Par \np, p \Tensor ((\top \Par q) \Tensor \bar{r}), r}
{\infer[(\Tensor)]{\vdash s, \np, p \Tensor ((\top \Par q) \Tensor \bar{r}), r}
{\infer[(\AX)]{\vdash \np, p}{} & 
\infer[(\Tensor)]{\vdash s, (\top \Par q) \Tensor \bar{r}, r}
{\infer[(\Par)]{\vdash s, \top \Par q}
{\infer[(\top)]{\vdash s, \top, q}{}} & \infer[(\AX)]{\vdash \bar{r}, r}{}}
}}} 
$$

Our proof of the $4 \Rightarrow 1$ implication for the restricted language
is essentially based on the ideas of~\citet{Schellinx}. We show that
if a sequent of the form $\vdash \widehat{\Gamma}^\bot, \widehat{B}$ is derivable
in $\Linlog$, then in each sequent in the derivation there is exactly one formula
of the form $\widehat{C}$, and all others are of the form $\widehat{C}^\bot$
(see Lemma~\ref{Lm:unique} below).
This means that all sequents in the $\Linlog$-derivation are actually translations
of Lambek sequents, and the derivation as a whole can be mapped onto a derivation
in $\LambekA$.

This technical lemma is proved using an extension of the $\natural$-counter 
by~\citet{PentusFmonov} to formulae of $\Linlog$ without $\Z$ and $\top$,
but possibly using additive and subexponential connectives (Pentus considers
only the multiplicative fragment of cyclic linear logic).
 
The $\natural$-counter is defined recursively:
 \begin{align*}
& \natural(p) = 0		&& \natural(A \Par B) = \natural(A) + \natural(B) - 1\\
& \natural(\np) = 1		&& \natural(A \Tensor B) = \natural(A) + \natural(B)\\
& \natural(\U) = 0		&& \natural(A \OR B) = \natural(A \AND B) = \natural(A)\\
& \natural(\bot) = 1	&&  \natural({?}^{s} A) = \natural({!}^s A) = \natural(A)
\end{align*}


If $\Gamma = E_1, \ldots, E_k$, then let $\natural(\Gamma) = \natural(E_1) + \ldots + \natural(E_k)$.

Then we establish the following properties of the $\natural$-counter:

\begin{lemma}\label{Lm:natural}
\begin{enumerate}
\item\label{Stm:posnatural} $\natural(\widehat{A}) = 0$;
\item\label{Stm:negnatural} $\natural(\widehat{A}^\bot) = 1$;
\item\label{Stm:ORnatural} if $A \OR B$ is of the form $\widehat{C}$ or
$\widehat{C}^\bot$, then $\natural(A) = \natural(B) = \natural(A \OR B)$;
\item\label{Stm:negsum} if each $A_i$ for $i = 1, \ldots, n$ is of the form $\widehat{C}$ or
$\widehat{C}^\bot$ and the sequent $\vdash A_1, \ldots, A_n$ is derivable in $\Linlog$, then
$\natural(A_1) + \ldots + \natural(A_n) = n - 1$;
\end{enumerate}
\end{lemma}

\begin{proof}
1. Induction on the structure of $A$:
\begin{align*}
& \natural(\widehat{p}) = \natural(p) = 0\\
& \natural(\widehat{\U}) = \natural(\U) = 0\\
& \natural(\widehat{A \cdot B}) = \natural(\widehat{A} \Tensor \widehat{B}) = \natural(\widehat{A}) + \natural(\widehat{B}) = 0 + 0 = 0\\
& \natural(\widehat{A \BS B}) = \natural(\widehat{A}^\bot \Par \widehat{B}) = \natural(\widehat{A}^\bot) + \natural(\widehat{B}) - 1 = 1 + 0 - 1 = 0\\
& \natural(\widehat{B \SL A}) = \natural(\widehat{B} \Par \widehat{A}^\bot) = \natural(\widehat{B}) + \natural(\widehat{A}^\bot) - 1 = 0 + 1 - 1 = 0\\
& \natural(\widehat{A \vee B}) = \natural(\widehat{A} \OR \widehat{B}) = \natural(\widehat{A}) = 0\\
& \natural(\widehat{A \wedge B}) = \natural(\widehat{A} \AND \widehat{B}) = \natural(\widehat{A}) = 0\\
& \natural(\widehat{{!}^s A}) = \natural({!}^s \widehat{A}) = \natural(\widehat{A}) = 0
\end{align*}

2. Induction on the structure of $A$:
\begin{align*}
& \natural(\widehat{p}^\bot) = \natural(\np) = 1\\
& \natural(\widehat{\U}^\bot) = \natural(\bot) = 1\\
& \natural((\widehat{A \cdot B})^\bot) = \natural(\widehat{B}^\bot \Par \widehat{A}^\bot) = \natural(\widehat{B}^\bot) + \natural(\widehat{A}^\bot) - 1 = 
1 + 1 - 1 = 1\\
& \natural((\widehat{A \BS B})^\bot) = \natural(\widehat{B}^\bot \Tensor \widehat{A}) = \natural(\widehat{B}^\bot) + \natural(\widehat{A}) = 1 + 0 = 1\\
& \natural((\widehat{B \SL A})^\bot) = \natural(\widehat{A} \Tensor \widehat{B}^\bot) = \natural(\widehat{A}) + \natural(\widehat{B}^\bot) = 0 + 1 = 1\\
& \natural((\widehat{A \vee B})^\bot) = \natural(\widehat{A}^\bot \AND \widehat{B}^\bot) = \natural(\widehat{A}^\bot) = 1\\
& \natural((\widehat{A \wedge B})^\bot) = \natural(\widehat{A}^\bot \OR \widehat{B}^\bot) = \natural(\widehat{A}^\bot) = 1\\
& \natural((\widehat{{!}^s A})^\bot) = \natural({?}^s \, \widehat{A}^\bot) = \natural(\widehat{A}^\bot) = 1
\end{align*}


3. If $A \OR B = \widehat{C}$, then $C = C_1 \vee C_2$, $A = \widehat{C}_1$, $B = \widehat{C}_2$, and
$\natural(A) = \natural(B) = 0$.

If $A \OR B = \widehat{C}^\bot$, then $C = C_1 \wedge C_2$, $A = \widehat{C}_1^\bot$, $B = \widehat{C}_2^\bot$, and
$\natural(A) = \natural(B) = 1$.



4. Induction on the derivation in $\Linlog$:

{\it Case 1,} $(\AX)$: $\natural(\widehat{A}) + \natural(\widehat{A}^\bot) = 0 + 1 = 1 = 2 - 1$, $n = 2$.

{\it Case 2,} $(\Tensor)$. Let $\Gamma$ include $n_1$ formulae and $\Delta$ include $n_2$ formulae. Then,
by induction hypothesis,
$\natural(\Gamma) + \natural(A) = n_1 + 1 - 1 = n_1$ and $\natural(B) + \natural(\Delta) = n_2 + 1 - 1 = n_2$.
Therefore, $\natural(\Gamma) + \natural(A \Tensor B) + \natural(\Delta) = \natural(\Gamma) + \natural(A) + \natural(B) + \natural(\Delta) = 
n_1 + n_2 = (n_1 + n_2 + 1) - 1 = n-1$.

{\it Case 3,} $(\Par)$. Let $\Gamma$ include $n_1$ formulae. Then, by induction hypothesis,
$\natural(A) + \natural(B) + \natural(\Gamma) = n_1 + 2 - 1 = n_1 + 1$. Therefore,
$\natural(A \Par B) + \natural(\Gamma) = \natural(A) + \natural(B) - 1 + \natural(\Gamma) = (n_1 + 1) - 1 = n-1$.

{\it Case 4,} $(\AND)$. Since $\natural(A_1 \AND A_2) = \natural(A_1)$, we have $\natural(A_1 \AND A_2) + \natural(\Gamma) = 
\natural(A_1) + \natural(\Gamma)$, which is $n-1$ by induction hypothesis.

{\it Case 5,} $(\OR)$. By Statement~\ref{Stm:ORnatural} of this Lemma, since $A_1 \AND A_2$ is of the form
$\widehat{C}$ or $\widehat{C}^\bot$, we have $\natural(A_1 \AND A_2) = \natural(A_i)$ for both $i = 1$ and $i = 2$.
Thus, $\natural(A_1 \OR A_2) + \natural(\Gamma) = \natural(A_i) + \natural(\Gamma)$, which is $n-1$ by induction hypothesis.

{\it Case 6,} $(\U)$: $\natural(\U) = 0 = 1 - 1$, $n = 1$.

{\it Case 7,} $(\bot)$. In this case $\Gamma$ contains $n-1$ formulae, by induction hypothesis
$\natural(\Gamma) = (n-1) - 1$, and $\natural(\bot) + \natural(\Gamma) = 1 + (n-1) - 1 = n - 1$.

{\it Case 8,} $(\top)$. Impossible, since $\top$ is neither of the form $\widehat{C}$, nor of the form
$\widehat{C}^\bot$.

{\it Case 9,} $(!)$. Adding ${!}^s$ doesn't alter the $\natural$-counter.

{\it Case 10,} $(?)$. Adding ${?}^s$ doesn't alter the $\natural$-counter.

{\it Case 11,} $(\WEAK)$. The new formula ${?}^s A$ couldn't be of the form $\widehat{C}$, therefore it is of the form
$\widehat{C}^\bot$. Hence, by Statement~\ref{Stm:negnatural} of this Lemma, $\natural({?}^s A) = 1$, and
$\natural({?}^s A) + \natural(\Gamma) = 1 + (n-1) - 1 = n - 1$.

{\it Case 12,} $(\NCONTR)$. Again, $\natural({?}^s A) = 1$, and
$\natural({?}^s A) + \natural(\Gamma) + \natural(\Delta) = \natural({?}^s A) + \natural(\Gamma) + \natural({?}^s A) + \natural(\Delta) - 1 = 
((n+1) - 1) - 1) = n-1$.

{\it Case 13,} $(\EXC)$. In this case $\natural({?}^s A) + \natural(\Gamma) + \natural(\Delta) = 
\natural(\Gamma) + \natural({?}^s A) + \natural(\Delta) = n - 1$ by induction hypothesis.


\end{proof}

\begin{lemma}\label{Lm:unique}
If each $A_i$ for $i = 1, \ldots, n$ is of the form $\widehat{C}$ or
$\widehat{C}^\bot$ and the sequent $\vdash A_1, \ldots, A_n$ is derivable in $\Linlog$, then
exactly one of $A_1, \ldots, A_n$ is of the form $\widehat{C}$, and all other are
of the form $\widehat{C}^\bot$.
\end{lemma}

\begin{proof}
Let our sequent include $k$ formulae of the form $\widehat{C}$ and $(n-k)$
formulae of the form $\widehat{C}^\bot$. Then, on one hand, 
$\natural(A_1) + \ldots + \natural(A_n) = n - 1$
by Statement~\ref{Stm:negsum} of the previous Lemma. On the other hand,
by Statements~\ref{Stm:posnatural} and~\ref{Stm:negnatural},
$\natural(A_1) + \ldots + \natural(A_n) = k \cdot 0 + (n - k) \cdot 1 = n - k$.
Thus, $n - k = n - 1$, whence $k = 1$.
\end{proof}

Now we are ready to prove Theorem~\ref{Th:embed}. 

\begin{proof}[Proof of Theorem~\ref{Th:embed}.]
\fbox{$4 \Rightarrow 1$} We proceed by induction on the derivation of $\vdash \widehat{\Pi}^\bot, \widehat{B}$ in $\Linlog$.
In our notation, we'll always put the formula of the form $\widehat{B}$ into the rightmost position (and use the cyclically
transformed versions of the rules, as shown above, see Section~\ref{S:Linlog}).


The most interesting case is the $(\Tensor)$ rule. If it yields the rightmost formula, $\widehat{B} = \widehat{E \cdot F} = 
\widehat{E} \Tensor \widehat{F}$, then the $(\Tensor)$ rule application transforms into
$(\to\cdot)$:

$$
\infer[(\Tensor_2)]{\vdash \widehat{\Gamma}^\bot, \widehat{\Delta}^\bot, 
\widehat{E} \Tensor \widehat{F}}
{\vdash \widehat{\Delta}^\bot, \widehat{E} & 
\vdash \widehat{\Gamma}^\bot, \widehat{F}}
\TRANSFORM
\infer[(\to\cdot)]{\Gamma, \Delta \to E \cdot F}{\Delta \to E & \Gamma \to F}
$$

If the $(\Tensor)$ rule yields a formula of the form $\widehat{A}^\bot$ from
$\widehat{\Pi}^\bot$, there are two possibilities:
$\widehat{A}^\bot$ is either $\widehat{E} \Tensor \widehat{F}^\bot = (\widehat{F \SL E})^\bot$ or
$\widehat{F}^\bot \Tensor \widehat{E} = (\widehat{E \BS F})^\bot$. Also one can apply the $(\Tensor)$ either in the
$(\Tensor_1)$ or in the $(\Tensor_2)$ form. This leads to four possible cases. Two of them are handled as follows:
$$
\infer[(\Tensor_2)]{\vdash \widehat{\Gamma}_1^\bot, \widehat{\Delta}^\bot, \widehat{E} \Tensor \widehat{F}^\bot, \widehat{\Gamma}_2^\bot, \widehat{B}}
{\vdash \widehat{\Delta}^\bot, \widehat{E} & 
\vdash
\widehat{\Gamma}_1^\bot, \widehat{F}^\bot, \widehat{\Gamma}_2^\bot, \widehat{B}}
\TRANSFORM
\infer[(\SL\to)]{\Gamma_2, F \SL E, \Delta, \Gamma_1 \to B}{\Delta \to E & \Gamma_2, F, \Gamma_1 \to B}
$$
$$
\infer[(\Tensor_1)]{\vdash \widehat{\Gamma}_1^\bot, \widehat{F}^\bot \Tensor \widehat{E}, 
\widehat{\Delta}^\bot, \widehat{\Gamma}_2^\bot, \widehat{B}}
{\vdash \widehat{\Gamma}_1^\bot, \widehat{F}^\bot, \widehat{\Gamma}_2^\bot, \widehat{B} &
\vdash \widehat{E}, \widehat{\Delta}^\bot}
\TRANSFORM
\infer[(\BS\to)]{\Gamma_2, E \BS F, \Delta, \Gamma_1 \to B}{\Delta \to E &
\Gamma_2, F, \Gamma_1 \to B}
$$
In the other two cases, we have the following:
$$
\infer[(\Tensor_1)]{\vdash \widehat{\Gamma}_1^\bot, \widehat{E} \Tensor \widehat{F}^\bot, \widehat{\Delta}^\bot, \widehat{\Gamma}_2^\bot, \widehat{B}}
{\vdash \widehat{\Gamma}_1^\bot, \widehat{E}, \widehat{\Gamma}_2^\bot, \widehat{B} 
& \vdash\widehat{F}^\bot, \widehat{\Delta}^\bot}
\qquad\text{\raisebox{1em}{or}}\qquad
\infer[(\Tensor_2)]{\vdash \widehat{\Gamma}_1^\bot, \widehat{\Delta}^\bot, 
\widehat{F}^\bot \Tensor \widehat{E}, \widehat{\Gamma}_2^\bot, \widehat{B}}
{\vdash \widehat{\Delta}^\bot,  \widehat{F}^\bot  & \vdash \widehat{\Gamma}_1^\bot, \widehat{E}, \widehat{\Gamma}_2^\bot, \widehat{B}}
$$
These situations violate Lemma~\ref{Lm:unique}, since in $\vdash\widehat{\Gamma}_1^\bot, \widehat{E}, \widehat{\Gamma}_2^\bot, \widehat{B}$
there are two formulae of the form $\widehat{C}$, and therefore this premise couldn't be derivable in $\Linlog$. Thus, these two cases are
impossible. 

All other rules are translated straightforwardly:
$$
\infer[(\Par)]{\vdash \widehat{\Gamma}^\bot, \widehat{E}^\bot \Par \widehat{F}}
{\vdash \widehat{\Gamma}^\bot, \widehat{E}^\bot, \widehat{F}}
\TRANSFORM
\infer[(\to\BS)]{\Gamma \to E \BS F}{E, \Gamma \to F}
$$
$$
\infer[(\Par)]{\vdash \widehat{\Gamma}^\bot, \widehat{F} \Par \widehat{E}^\bot}
{\vdash \widehat{E}^\bot, \widehat{\Gamma}^\bot, \widehat{F}}
\TRANSFORM
\infer[(\to\BS)]{\Gamma \to F \SL E}{\Gamma, E \to F}
$$
$$
\infer[(\Par)]{\vdash \widehat{\Gamma}_1^\bot, \widehat{E}^\bot \Par \widehat{F}^\bot,
\widehat{\Gamma}_2^\bot, \widehat{B}}
{\vdash \widehat{\Gamma}_1^\bot, \widehat{E}^\bot, \widehat{F}^\bot,
\widehat{\Gamma}_2^\bot, \widehat{B}}
\TRANSFORM
\infer[(\cdot\to)]{\Gamma_2, F \cdot E, \Gamma_1 \to B}{\Gamma_2, F, E, \Gamma_1 \to B}
$$
$$
\infer[(\AND)]{\vdash \widehat{\Gamma}^\bot, \widehat{E}_1 \AND \widehat{E}_2}
{\vdash \widehat{\Gamma}^\bot, \widehat{E}_1 & 
\vdash \widehat{\Gamma}^\bot, \widehat{E}_2}
\TRANSFORM
\infer[(\to\wedge)]{\Gamma \to E_1 \wedge E_2}{\Gamma \to E_1 & \Gamma \to E_2}
$$
$$
\infer[(\AND)]{\vdash \widehat{\Gamma}_1^\bot, \widehat{F}_1^\bot \AND \widehat{F}_2^\bot,
\widehat{\Gamma}_2^\bot, \widehat{B}}
{\vdash \widehat{\Gamma}_1^\bot, \widehat{F}_1^\bot,
\widehat{\Gamma}_2^\bot, \widehat{B} &
\vdash \widehat{\Gamma}_1^\bot, \widehat{F}_2^\bot,
\widehat{\Gamma}_2^\bot, \widehat{B}}
\TRANSFORM
\infer[(\vee\to)]{\Gamma_2, F_1 \vee F_2, \Gamma_1 \to B}
{\Gamma_2, F_1, \Gamma_1 \to B & \Gamma_2, F_2, \Gamma_1 \to B}
$$
$$
\infer[(\OR)]{\vdash \widehat{\Gamma}^\bot, \widehat{E}_1 \OR \widehat{E}_2}
{\vdash \widehat{\Gamma}^\bot, \widehat{E}_i}
\TRANSFORM
\infer[(\to\vee)]{\Gamma \to E_1 \vee E_2}{\Gamma \to E_i}
$$
$$
\infer[(\OR)]{\vdash \widehat{\Gamma}_1^\bot, \widehat{F}_1^\bot \OR \widehat{F}_2^\bot,
\widehat{\Gamma}_2^\bot, \widehat{B}}
{\vdash \widehat{\Gamma}_1^\bot, \widehat{F}_i^\bot,
\widehat{\Gamma}_2^\bot, \widehat{B}}
\TRANSFORM
\infer[(\wedge\to)]{\Gamma_2, F_1 \wedge F_2, \Gamma_1 \to B}
{\Gamma_2, F_i, \Gamma_1 \to B}
$$
$$
\infer[(\U)]{\vdash \U}{} \TRANSFORM \infer[(\to\U)]{\to\U}{}
$$
$$
\infer[(\bot)]{\vdash \widehat{\Gamma}_1^\bot, \bot, \widehat{\Gamma}_2^\bot,
\widehat{B}}{\vdash \widehat{\Gamma}_1^\bot, \widehat{\Gamma}_2^\bot,
\widehat{B}}
\TRANSFORM
\infer[(\U\to)]{\Gamma_2, \U, \Gamma_1 \to B}{\Gamma_2, \Gamma_1 \to B}
$$
The $(\top)$ rule cannot be applied, since $\top$ is neither of the form 
$\widehat{C}$, nor of the form $\widehat{C}^\bot$.
$$
\infer[(!)]{\vdash ?^{s_1} \widehat{A}_1^\bot, \ldots, {?}^{s_n} \widehat{A}_n^\bot,
{!}^{s} \widehat{B}}{\vdash ?^{s_1} \widehat{A}_1^\bot, \ldots, {?}^{s_n} \widehat{A}_n^\bot,
\widehat{B}}
\TRANSFORM
\infer[(\to{!})]{{!}^{s_n} A_n, \ldots, {!}^{s_1} A_1 \to {!}^s B}
{{!}^{s_n} A_n, \ldots, {!}^{s_1} A_1 \to B}
$$

\fbox{$1 \Rightarrow 2$} Trivial: allowing the cut rule doesn't invalidate
cut-free derivations.

\fbox{$2 \Rightarrow 3$} Straightforward induction on the derivation in 
$\LambekA + (\CUT)$. The cut rule is translated as follows:
$$
\infer[(\CUT)]{\Gamma_1, \Pi, \Gamma_2 \to C}
{\Pi \to A & \Gamma_1, A, \Gamma_2 \to C}
\TRANSFORM
\infer[(\CUT)]{\vdash \widehat{\Pi}^\bot,
\widehat{\Gamma}_1^\bot, \widehat{C}, \widehat{\Gamma}_2^\bot}
{\vdash \widehat{\Pi}^\bot, \widehat{A} & 
\vdash \widehat{A}^\bot, \widehat{\Gamma}_1^\bot, \widehat{C}, \widehat{\Gamma}_2^\bot}
$$

For translating other rules, one simply reverses arrows in 
the proof of the $4 \Rightarrow 1$ implication (see above).

\fbox{$3 \Rightarrow 4$} Follows from cut elimination in $\Linlog$ 
(Theorem~\ref{Th:cutelim}).
\end{proof}

\section{Cut vs. Contraction}\label{S:CutVsContr}

The contraction rules of $\LambekA$ and $\Linlog$ are {\em non-local,} {\em i.e.,}
they can take formulae for contraction from distant places of the sequent. In the
presence of exchange (permutation) rules, non-local contraction rules are equivalent
to {\em local} ones, that contract two neighbour copies of the same formula
marked with an appropriate subexponential: 
$$
\infer[(\CONTR),\text{ for $\LambekA$;}]
{\Gamma_1, {!}^s A, \Gamma_2 \to C}
{\Gamma_1, {!}^s A, {!}^s A, \Gamma_2 \to C}
\qquad
\infer[(\CONTR),\text{ for $\Linlog$.}]
{\vdash {?}^s A, \Gamma}
{\vdash {?}^s A, {?}^s A, \Gamma}
$$

If the subexponential doesn't allow exchange ($s \in \Cc - \Ec$), however, 
cut elimination {\bf fails.}


\begin{theorem}\label{Th:cutfail}
The extension of the Lambek calculus with a unary connective ${!}$ axiomatised by
rules $({!} \to)$, $(\to {!})$, $(\CONTR)$, and, optionally,
  $(\WEAK)$ does not admit $(\CUT)$.
\end{theorem}

\begin{proof}
One can take the following sequent as a counter-example:
$$ r \SL q, {!}p, {!}(p \BS q), q \BS s \to r \cdot s. $$
This sequent has a proof with $(\CUT)$:
$$
\infer[(\CUT)]{r \SL q, {!}p, {!}(p \BS q), q \BS s \to r \cdot s}
{\infer[(\to{!})]{{!}p, {!}(p \BS q) \to {!}q}
{\infer[({!}\to)\mbox{ twice}]{{!}p, {!}(p \BS q) \to q}{\infer[(\BS\to)]{p, p \BS q \to q}{p \to p & q \to q}}} &
\infer[(\CONTR)]{r \SL q, {!}q, q \BS s \to r \cdot s}
{\infer[({!} \to) \mbox{ twice}]{r \SL q, {!}q, {!}q, q \BS s \to r \cdot s}
{
\infer[(\SL\to)]{r \SL q, q, q, q \BS s \to r \cdot s}
{q \to q & \infer[(\BS\to)]{r, q, q \BS s \to r \cdot s}{q \to q & \infer[(\to\cdot)]{\vphantom{\BS} r, s \to r \cdot s}{r \to r & s \to s}}}
}}
}
$$
but doesn't have a cut-free proof. In order to verify the latter, we notice that due to
subformula and polarity properties, the only rules that can be applied are
$(\SL\to)$, $(\BS\to)$, $(\to\cdot)$, $({!}\to)$, and $(\CONTR)$. Moreover,
since ${!}$ appears only on the top level, the rules operating ${!}$ can be moved
to the very bottom of the proof (this is actually a small {\em focusing} instance
here). These rules would be applied to a (pure Lambek) sequent of the form
$$
r \SL q, p, \ldots, p, (p \BS q), \ldots, (p \BS q), q \BS s \to r \cdot s,
$$
but an easy proof search attempt shows that none of these sequents is derivable in the Lambek calculus.
\end{proof}

This failure of cut elimination of the calculus with $(\CONTR)$ motivates 
the usage of the non-local version of contraction, $(\NCONTR_{1,2})$.


With non-local contraction, the sequent used in the proof of Theorem~\ref{Th:cutfail} obtains a
cut-free proof:
$$
\infer[(\CONTR)]
{r \SL q, {!}p, {!}(p \BS q), q \BS s \to r \cdot s}
{\infer[(\NCONTR_2)]{r \SL q, {!}p, {!}p, {!}(p \BS q), q \BS s \to r \cdot s}
{\infer[({!}\to)\text{ \small 4 times}]{r \SL q, {!}p, {!}(p \BS q), {!}p, {!}(p \BS q), q \BS s \to r \cdot s}
{\infer[(\BS\to)\text{ \small twice}]{r \SL q, p, p \BS q, p, p \BS q, q \BS s \to r \cdot s}
{p \to p & p \to p &
\infer[(\SL\to)]{r \SL q, q, q, q \BS s \to r \cdot s}
{q \to q & \infer[(\BS\to)]{r, q, q \BS s \to r \cdot s}{q \to q & \infer[(\to\cdot)]{\vphantom{\BS} r, s \to r \cdot s}{r \to r & s \to s}}}
}}}}
$$

This counter-example can also be translated into $\Linlog$ using the embedding
of $\Lambek$ into $\Linlog$ (see Section~\ref{S:embed}).



\section{Undecidability of $\Lambek$}\label{S:undec}


In the view of Corollary~\ref{Cor:embed}, we prove {\em lower} complexity bounds for fragments of $\Lambek$ and
{\em upper} ones for fragments of $\Linlog$.


\begin{theorem}\label{Th:undec}
If $\Cc \ne \varnothing$ ({\em i.e.,} at least one subexponential allows the non-local contraction rule),
then the derivability problem in $\Lambek^\U$ is undecidable. 
\end{theorem}

The proof is follows the line presented in~\citet{KanKuzSceFCT}, using
ideas from \citet{LMSS}, \citet{Kanazawa}, and \citet{deGroote}. In the latter three
papers, undecidability is established for non-commutative propositional linear logic
systems equipped with an exponential that allows all structural rules (contraction, weakening,
and exchange), as $\LambekE$ defined below. The difference of our setting is that here only contraction is guaranteed and
exchange and weakening are optional.

The undecidability proof is based on encoding word rewriting (semi-Thue)
systems~\citep{Thue}. 
A word rewriting system over alphabet $\Af$ is a finite set $P$ of pairs of
words over $\Af$. Elements of $P$ are called {\em rewriting rules} and are applied
as follows: if $\langle \alpha, \beta \rangle \in P$, then $\eta \, \alpha \, \theta
\Rightarrow \eta \, \beta \, \theta$ for arbitrary (possibly empty) words $\eta$ and
$\theta$ over $\Af$. The relation $\Rightarrow^*$ is the reflexive transitive closure
of $\Rightarrow$.

The following classical result appears in works of~\citet{markov47dan}
and~\citet{post47jsl}.

\begin{theorem}\label{Th:MarkovPost}
There exists a word rewriting system $P$ such that the set 
$\{ \langle \gamma, \delta \rangle \mid \gamma \Rightarrow^* \delta \}$
is r.e.-complete (and therefore undecidable).~{\rm\citep{markov47dan,post47jsl}}
\end{theorem}

In our encoding we'll actually need the weakening rule. However, our subexponential doesn't
necessarily enjoy it.
To simulate weakening, we use the unit constant:
actually, the $(\U\to)$ rule is weakening, but for $\mathbf{1}$ rather than ${!}A$.

Let $P$ be the word rewriting system from Theorem~\ref{Th:MarkovPost} and consider
all elements of $\Af$ as variables of the Lambek calculus. We convert rewriting rules
of $P$ into Lambek formulae in the following way:
$$
\mathcal{B} = \{ (u_1 \cdot \ldots \cdot u_k) \SL (v_1 \cdot \ldots \cdot v_m) \mid \langle u_1 \ldots u_k, v_1 \ldots v_m \rangle \in P \}.
$$
If $\mathcal{B} = \{ B_1, \ldots, B_n \}$ (we can take any ordering of $\mathcal{B}$), let
$$
\Phi = \mathbf{1} \SL {!}^s B_1, {!}^s B_1, \ldots, \mathbf{1} \SL {!}^s B_n, {!}^s B_n.
$$
Finally, we consider a {\em theory} (finite set of sequents) $\mathcal{T}$ associated with $P$:
$$
\mathcal{T} = \{ v_1, \ldots, v_m \to u_1 \cdot \ldots \cdot u_k \mid \langle u_1 \ldots u_k, v_1 \ldots v_m \rangle \in P \}.
$$
When talking about {\em derivability from theory} $\mathcal{T}$, we use the rules of the original Lambek calculus,
including cut.

Now let $s \in \Cc$ be the label of the subexponential that allows non-local contraction
(and, possibly, also weakening and/or exchange). We also consider, as a technical
tool, the extension of the Lambek calculus 
with an exponential modality ${!}$ that
allows all three structural rules, contraction, weakening, and exchange. We denote
this auxiliary calculus by $\LambekE^\U$. 

In our framework, $\LambekE^\U$
is $\mathrm{SLC}^\U_{\Sigma_0}$ with a trivial
subexponential signature $\Sigma_0 = \langle \Ic_0, {\preceq_0}, \Wc_0,
\Cc_0, \Ec_0 \rangle$, where $\Ic_0 = \Wc_0 = \Cc_0 = \Ec_0 = \{ s_0 \}$, 
$\preceq_0$ is trivial, and ${!}^{s_0}$ is denoted
by ${!}$. Thus, $\LambekE^\U$ enjoys all proof-theoretical properties of
$\Lambek^\U$, in particular, cut elimination (Corollary~\ref{Cor:Lambekcutelim}).

For $\mathcal{B} = \{ B_1, \ldots, B_n \}$, let $\Gamma = {!}B_1, \ldots, {!}B_n$.

\begin{lemma}\label{Lm:encoding}
Let $\gamma = a_1 \ldots a_l$ and $\delta = b_1 \ldots b_k$ be arbitrary words over $\Af$.
Then the following are equivalent:
\begin{enumerate}
\item $\gamma \Rightarrow^* \delta$;
\item the sequent $\Phi, b_1, \ldots, b_k \to a_1 \cdot \ldots \cdot a_l$ is derivable in 
$\Lambek^\U$;
\item the sequent $\Gamma,  b_1, \ldots, b_k \to a_1 \cdot \ldots \cdot a_l$ is derivable in
$\LambekE^\U$;
\item the sequent $b_1, \ldots, b_k \to a_1 \cdot\ldots\cdot a_l$ is derivable from $\mathcal{T}$.
\end{enumerate}
\end{lemma}

\begin{proof}
\fbox{$1 \Rightarrow 2$}
Proceed by induction on $\Rightarrow^*$. The base case 
($\gamma \Rightarrow^* \gamma$) is handled as follows:
$$
\infer[(\SL\to) \text{ \small $n$ times}]{\strut \mathbf{1} \SL {!}^s B_1, {!}^s B_1, \ldots, \mathbf{1} \SL {!}^s B_n, {!}^s B_n, 
a_1, \ldots, a_l \to a_1 \cdot \ldots \cdot a_l}
{B_1 \to B_1 & \ldots & B_n \to B_n & 
\infer[(\mathbf{1}\to) \text{ \small $n$ times}]{\strut \mathbf{1}, \ldots, \mathbf{1}, a_1, \ldots, a_l \to a_1 \cdot \ldots \cdot a_l}
{\infer[(\to\cdot) \text{ \small $(l-1)$ times}]{ \strut a_1 \ldots a_m \to a_1 \cdot\ldots\cdot a_l}{a_1 \to a_1 & \ldots & a_m \to a_l }}}
$$

For the induction step, consider the last step of $\Rightarrow^*$:
$$
\gamma \Rightarrow^* \eta\, u_1 \ldots u_k \, \theta \Rightarrow \eta\, v_1 \ldots v_m \, \theta.
$$
Then, since ${!}^s((u_1 \cdot \ldots \cdot u_k) \SL (v_1 \cdot \ldots \cdot v_m))$ is in $\Phi$
and $s \in \Cc$, we enjoy
the following derivation:

$$
\infer[(\NCONTR_1)]
{\Phi, \eta, v_1, \ldots, v_m, \theta \to a_1 \cdot\ldots\cdot a_l}
{\infer[({!}\to)]{\Phi, \eta, {!}^s((u_1 \cdot \ldots \cdot u_k) \SL (v_1 \cdot \ldots \cdot v_m)), v_1, \ldots, v_m, \theta \to a_1 \cdot\ldots\cdot a_l}
{\infer[(\SL\to)]{\Phi, \eta, (u_1 \cdot \ldots \cdot u_k) \SL (v_1 \cdot \ldots \cdot v_m), v_1, \ldots, v_m, \theta \to a_1 \cdot\ldots\cdot a_l}
{\infer[(\to\cdot) \text{  \small $(m-1)$ times}]{\vphantom{\Gamma} v_1, \ldots, v_m \to v_1 \cdot \ldots \cdot v_m}{v_1 \to v_1 & \ldots & v_m \to v_m} & 
\infer[(\cdot\to) \text{ \small $(k-1)$ times}]{\Phi, \eta, u_1 \cdot \ldots \cdot u_k, \theta \to a_1 \cdot\ldots\cdot a_l}
{\Phi, \eta, u_1, \ldots, u_k, \theta \to a_1\cdot\ldots\cdot a_l}}}}
$$

The sequent $\Phi, \eta, u_1, \ldots, u_k, \theta \to a_1\cdot\ldots\cdot a_l$ is derivable by induction hypothesis.

\fbox{$2 \Rightarrow 3$}
For each formula $B_i \in \mathcal{B}$ the sequent $\to \mathbf{1} \SL {!}B_i$ is derivable
in $\LambekE^\U$
using the weakening rule:
$$
\infer[(\to\SL)]{\to \mathbf{1} \SL {!}B_i}
{\infer[(\WEAK)]{{!}B_i \to \mathbf{1}}{\to\mathbf{1}}}
$$
Then we notice that, since ${!}$ in $\LambekE^\U$ obeys all the rules for
${!}^s$ in $\Lambek^\U$, the sequent 
$\Phi', b_1, \ldots, b_k \to a_1 \cdot \ldots \cdot a_l$, where
$\Phi'$ is the result of replacing ${!}^s$ by ${!}$ in $\Phi$,
 is derivable in $\LambekE^\U$.
Then we apply $(\CUT)$ to remove formulae of the form $\U \SL {!}B_i$ from
$\Phi'$. 
This transforms $\Phi'$ into $\Gamma$ and yields derivability of
$\Gamma, b_1, \ldots, b_k \to a_1 \cdot \ldots \cdot a_l$ in $\LambekE^\U$.

\fbox{$3 \Rightarrow 4$}
Consider the cut-free derivation of $\Gamma,  b_1, \ldots, b_k \to a_1 \cdot \ldots \cdot a_l$
(as shown above, $\LambekE^\U$ enjoys cut elimination).
Remove all formulae of the form
${!}E$ from the left-hand sides of the sequents in this derivation. This transformation doesn't affect rules not operating with ${!}$,
they remain valid. Applications of structural rules ($(\NCONTR_{1,2}$, $(\EXC_{1,2})$, $(\WEAK)$)  do not alter the sequent.
The only non-trivial case is $({!}\to)$. Since all formulae of the form ${!}E$ come from $\Gamma$ (due to the subformula property
of the cut-free derivation), the only possible case is the following one:

$$
\infer{\Delta_1, \Delta_2 \to C}
{\Delta_1, (u_1 \cdot \ldots \cdot u_k) \SL (v_1 \cdot \ldots \cdot v_m), \Delta_2 \to C}
$$
($(u_1 \cdot \ldots \cdot u_k) \SL (v_1 \cdot \ldots \cdot v_m)$ transforms into an invisible
${!}((u_1 \cdot \ldots \cdot u_k) \SL (v_1 \cdot \ldots \cdot v_m))$).
This application is simulated using an extra axiom from the theory $\mathcal{T}$ that we're allowed to use:
$$
\infer[(\CUT)]{\Delta_1, \Delta_2 \to C}
{\infer[(\to\SL)]{\to (u_1 \cdot \ldots \cdot u_k) \SL (v_1 \cdot \ldots \cdot v_m)}
{ 
\infer[(\cdot\to) \text{ \small $(k-1)$ times}]{\vphantom{\Gamma}  v_1 \cdot \ldots \cdot v_m \to u_1 \cdot \ldots u_k}
{\vphantom{\Gamma} v_1, \ldots, v_m \to u_1 \cdot \ldots \cdot u_k}} & 
\Delta_1, (u_1 \cdot \ldots \cdot u_k) \SL (v_1 \cdot \ldots \cdot v_m), \Delta_2 \to C}
$$
The sequent $v_1, \ldots, v_m \to u_1 \cdot \ldots \cdot u_k$ belongs to $\mathcal{T}$.

\fbox{$4 \Rightarrow 1$}
Derivations from $\mathcal{T}$ essentially need the cut rule. However, if one tries to apply the standard
cut elimination procedure, all the cuts move directly to new axioms from $\mathcal{T}$ (this procedure
is called {\em cut normalization}). This yields a weak form of subformula property: any formula appearing
in a normalized derivation is a subformula either of $\mathcal{T}$, or of the goal sequent.
Since both $\mathcal{T}$ and $b_1, \ldots, b_k \to a_1 \cdot\ldots\cdot a_l$ include only variables
and the product operation, $\cdot$, rules for other connectives are never applied in the normalized derivation.
For simplicity we omit parentheses and the ``$\cdot$'' symbols, and the rules get formulated in the
following way:
$$
\infer[(\to\cdot)]{\beta_1 \beta_2 \to \alpha_1 \alpha_2}{\beta_1 \to \alpha_1 & \beta_2 \to \alpha_2} \qquad
\infer[(\CUT)]{\eta\beta\theta \to \gamma}{\alpha \to \beta & \eta\alpha\theta \to \gamma}
$$
(the $(\cdot\to)$ rule becomes trivial), and the axioms are $\alpha \to \alpha$ and rewriting rules from
$P$ with the arrows inversed.

One can easily check the following:
\begin{itemize}
\item if $\alpha_1 \Rightarrow^* \beta_2$ and $\alpha_2 \Rightarrow^* \beta_2$, then
$\alpha_1 \alpha_2 \Rightarrow^* \beta_1 \beta_2$;
\item if $\alpha \Rightarrow^* \beta$ and $\gamma \Rightarrow^* \eta\alpha\theta$, then
$\gamma \Rightarrow^* \eta\beta\theta$.
\end{itemize}
Then, by induction on the derivation, we get $a_1 \ldots a_l \Rightarrow^* b_1 \ldots b_k$,
{\em i.e.,} $\gamma \Rightarrow^* \delta$.
\end{proof}

One could get rid of the unit constant, using the technique from~\citep{kuznetsov11}:

\begin{lemma}\label{Lm:unit}
Let $q$ be a fresh variable and let $\widetilde{\Gamma} \to \widetilde{C}$ be the sequent
$\Gamma \to C$ with $\mathbf{1}$ replaced with $q \SL q$ and every variable $p_i$ replaced
with $(q \SL q) \cdot p_i \cdot (q \SL q)$. Then $\Gamma \to C$ is derivable if and only
if $\widetilde{\Gamma} \to \widetilde{C}$ is derivable.
\end{lemma}


\begin{proof}
The $(\mathbf{1}\to)$ rule can be interchanged with any rule applied before. Thus one
can place all applications of $(\mathbf{1}\to)$ directly after axioms. All other rules,
except $(\mathbf{1}\to)$, remain valid after the replacements. Axioms 
with $(\U\to)$ applied 
are sequents of the form 
$\mathbf{1}, \ldots, \mathbf{1}, p_i, \mathbf{1}, \ldots, \mathbf{1} \to p_i$ or 
$\mathbf{1}, \ldots, \mathbf{1} \to \mathbf{1}$. After the replacements they become
derivable sequents $q \SL q, \ldots, q \SL q, (q \SL q) \cdot p_i \cdot (q \SL q), q \SL q, \ldots, q \SL q \to q \SL q$
and $q \SL q, \ldots, q \SL q \to q \SL q$. This justifies the ``only if'' part.

For the ``if'' part, we start with $\widetilde{\Gamma} \to \widetilde{C}$ and substitute
$\mathbf{1}$ for $q$ (substitution of arbitrary formulae for variables is legal in
the $\Lambek^\U$). Since $(\mathbf{1} \SL \mathbf{1})$ is equivalent to $\mathbf{1}$ and
$(\mathbf{1} \SL \mathbf{1}) \cdot p_i \cdot (\mathbf{1} \SL \mathbf{1})$ is equivalent to $p_i$,
the result of this substitution is equivalent to $\Gamma \to C$, whence this sequent is derivable.
\end{proof}



This yields the following theorem:

\begin{theorem}
If $\Cc \ne \varnothing$, then the derivability problem in $\Lambek$ is
undecidable.
\end{theorem}


Finally, $\LambekA$ and $\Linlog$, being conservative extensions of $\Lambek$, is also undecidable:

\begin{corollary}
If $\Cc \ne \varnothing$, then the derivability problem in $\LambekA$ is undecidable.
\end{corollary}

\begin{corollary}
If $\Cc \ne \varnothing$, then the derivability problem in $\Linlog$ is undecidable.
\end{corollary}

\section{Decidability of Systems without Contraction}\label{S:decid}

The non-local contraction rule plays a crucial role in our undecidability proof presented in the previous section.
If there is no subexponential that allows contraction ({\em i.e.,} $\Cc = \varnothing$), the derivability problem
becomes decidable:

\begin{theorem}\label{Th:decid}
If $\Cc = \varnothing$, then the decidability problem for $\Linlog$ belongs to PSPACE and the
decidability problem for $\SMCLL$ belongs to NP. Hence, both problems are
algorithmically decidable.
\end{theorem}

(Recall that $\Linlog$ is the full cyclic linear logic with subexponentials and
$\SMCLL$ is the system without additive constants and connectives, $\top$, $\Z$, $\AND$ and $\OR$.)

\begin{proof}
By Theorem~\ref{Th:cutelim}, we consider only cut-free derivations.
Since contraction is never applied, each rule, except exchange,
introduces at least one new connective into the sequent
(weakening and the $(\top)$ axiom can introduce whole subformulae at once, all other rules introduce
exactly one connective per rule).
 Thus, in the situation without additive conjunction (in $\SMCLL$) these connectives can be disjointly traced down to the goal sequent, and
each rule application can be associated with a unique connective occurrence in the goal sequent. 
For exchange rules, we consider several consequent applications of $(\EXC)$, possibly for different
${?}^s A$, as one rule. Correctness of such a joint exchange rule application can still be checked in polynomial time.
After this joining, each exchange rule is followed by another rule or yields the goal sequent, therefore
applications of $(\EXC)$ give not more than a half of the total number of rules applied in the derivation.
Thus, the size of a cut-free derivation in $\SMCLL$, for $\Sigma$ with $\Cc = \varnothing$, is linearly
bounded by the size of the goal sequent. Since checking correctness of a derivation can be done in
polynomial time, this derivation serves as an NP witness, so the derivability problem for
$\SMCLL$, for $\Sigma$ with $\Cc = \varnothing$, belongs to the NP class.

For the whole $\Linlog$ system, we
follow the strategy by~\citet[Section 2.1]{LMSS}. Namely, we show that the height
of a cut-free derivation tree (again, with joined exchange rules) is linear w.r.t.\ the size of the goal sequent.
This follows from the fact that on a path from the goal sequent to an axiom leaf in the derivation tree each rule
either introduces new connectives into the goal sequent or is an exchange rule. Therefore, the length of
such a path is linearly bounded by the size of the goal sequent. (On the other hand the size of the whole derivation
tree could be exponential, because the $(\AND)$ rule copies the same formulae into different branches.)
A derivation tree of polynomial height can be guessed and
checked by a non-deterministic Turing machine with polynomially bounded space, using the depth-first procedure~\citep[Section 2.1]{LMSS}.
This establishes the fact that the derivability problem for $\Linlog$, for $\Sigma$ with $\Cc = \varnothing$, belongs to NPSPACE, which is equal to
deterministic PSPACE by Savitch's theorem~\citep{Savitch}.
\end{proof}

By Corollary~\ref{Cor:embed}, we also get decidability results for the corresponding Lambek systems:

\begin{corollary}
If $\Cc = \varnothing$, then the decidability problem for $\LambekA$ belongs to PSPACE and the
decidability problem for $\Lambek$ belongs to NP. Hence, both systems are
algorithmically decidable.
\end{corollary}

Notice that these complexity bounds are exact, since even without subexponentials the derivability problems in
the purely multiplicative Lambek calculus is NP-complete~\citep{PentusNP} and the derivability problem in
the multiplicative-additive Lambek calculus is PSPACE-complete~\citep{KanovichKazimierz} (see also~\citet{Kanazawa}). 

\section{Conclusions and Future Work}\label{S:future}

In this paper we have considered two systems of non-commutative linear logic---the multiplicative-additive Lambek calculus
and cyclic propositional linear logic---and extended them with subexponentials.
For these extended systems, we've proved cut elimination and shown that the first system can be conservatively embedded
into the second one. We've also shown that, for cut elimination to hold, the contraction rule should be in the non-local form.
Finally, we've established exact algorithmic complexity estimations. Namely, at least one subexponential that allows contraction
makes the system undecidable. On the other hand, subexponentials that don't allow contraction do not increase complexity in comparison
with the original system without subexponentials: it is still NP for multiplicative systems and PSPACE for multiplicative-additive ones.

A natural step to take from here is to investigate focused~\citep{Andreoli} proof systems with non-commutative subexponentials. 
This would open a number of possibilities such as the development of logical frameworks with non-commutative 
subexponentials. Such frameworks have been used, for example, by~\citet{DBLP:conf/lics/PfenningS09} for
the specification of evaluation strategies of functional programs. While their focused proof system contained a single unbounded,
a single bounded and a single non-commutative modalities, focused proof systems with commutative and non-commutative subexponentials 
would allow for any number of modalities allowing the encoding of an even wider range of systems. 
Such investigation is left for future work.

In our undecidability proof, we encoded semi-Thue systems in $\Lambek$, using only three connectives, $\SL$ (one can dually
use $\BS$, of course), $\cdot$, and ${?}^s$ (where $s \in \Cc$). The language can be further restricted to
$\SL$ and ${?}^s$, without $\cdot$, by using a more sophisticated encoding by~\citet{BuszkoZML}, 
see~\citet{KanKuzSceFG}. The number of variables used in the construction could be also reduced to one variable
using the technique by~\citet{KanovichNeutral}. We leave the details of these restrictions for future work.

On the other hand, if we allow subexponentials with contraction to be applied only to variables (${?}^s p$) or
to formulae without $\cdot$ of implication depth 1 (for example, ${?}^s (p \SL q)$), the derivability problem
probably becomes decidable, which would be quite nice for linguistic applications. We leave this as an open
question for future studies.

For extensions of the Lambek calculus, another interesting question, besides decidability and algorithmic complexity, is the
generative power of categorial grammars based on these extensions. Original Lambek grammars generate precisely context-free
languages~\citep{PentusCF}. On the other hand, it actually follows from our undecidability proof that grammars based on $\Lambek$,
where at least one subexponential in $\Sigma$ allows contraction ($\Cc \ne \varnothing$), can generate an arbitrary recursively
enumerable language. For decidable fragments ({\em e.g.,} when $\Cc = \varnothing$, or subexponentials allowing contraction
are somehow restricted syntactically), however, determining the class of languages generated by corresponding grammars is left
for future research.


\bibliographystyle{abbrvnat}
\bibliography{subexponential.bib}

\providecommand{\noopsort}[1]{}
\begin{thebibliography}{51}
\providecommand{\natexlab}[1]{#1}
\providecommand{\url}[1]{\texttt{#1}}
\expandafter\ifx\csname urlstyle\endcsname\relax
  \providecommand{\doi}[1]{doi: #1}\else
  \providecommand{\doi}{doi: \begingroup \urlstyle{rm}\Url}\fi

\bibitem[Abrusci(1990)]{Abrusci}
V.~M. Abrusci.
\newblock {A comparison between Lambek syntactic calculus and intuitionistic
  linear logic}.
\newblock \emph{Zeitschr. math. Logik Grundl. Math. (Math. Logic Q.)},
  36:\penalty0 11--15, 1990.

\bibitem[Andreoli(1992)]{Andreoli}
J.-M. Andreoli.
\newblock {Logic programming with focusing proofs in linear logic}.
\newblock \emph{J. Logic Comput.}, 2\penalty0 (3):\penalty0 297--347, 1992.

\bibitem[{\noopsort{Benthem}}{van Benthem}(1991)]{vanBenthemLIA}
J.~{\noopsort{Benthem}}{van Benthem}.
\newblock \emph{{Language in action: categories, lambdas and dynamic logic}}.
\newblock North Holland, Amsterdam, 1991.

\bibitem[Bra{\"u}ner and de~Paiva(1998)]{dePaiva}
T.~Bra{\"u}ner and V.~de~Paiva.
\newblock {A formulation of linear logic based on dependency relations}.
\newblock In \emph{Proc. CSL '97}, volume 1414 of \emph{LNCS}, pages 129--148.
  Springer, 1998.

\bibitem[Buszkowski(1982)]{BuszkoZML}
W.~Buszkowski.
\newblock {Some decision problems in the theory of syntactic categories}.
\newblock \emph{Zeitschr. Math. Logik Grundl. Math. (Math. Logic Q.)},
  28:\penalty0 539--548, 1982.

\bibitem[Buszkowski(2010)]{BuszkoLCSL}
W.~Buszkowski.
\newblock {Lambek calculus and substructural logics}.
\newblock \emph{Linguistic Analysis}, 36\penalty0 (1--4):\penalty0 15--48,
  2010.

\bibitem[Chaudhuri(2010)]{Chaudhuri}
K.~Chaudhuri.
\newblock {Classical and intuitionistic subexponential logics are equally
  expressive}.
\newblock In \emph{Proc. CSL '10}, volume 6247 of \emph{LNCS}, pages 185--199.
  Springer, 2010.

\bibitem[Danos et~al.(1993)Danos, Joinet, and Schellinx]{danos93kgc}
V.~Danos, J.-B. Joinet, and H.~Schellinx.
\newblock {The structure of exponentials: Uncovering the dynamics of linear
  logic proofs}.
\newblock In \emph{Kurt G{\"o}del Colloquium}, volume 713 of \emph{LNCS}, pages
  159--171. Springer, 1993.

\bibitem[de~Groote(2005)]{deGroote}
P.~de~Groote.
\newblock {On the expressive power of the Lambek calculus extended with a
  structural modality}.
\newblock In \emph{{Language and Grammar}}, volume 168 of \emph{CSLI Lect.
  Notes}, pages 95--111. 2005.

\bibitem[Gentzen(1935)]{Gentzen}
G.~Gentzen.
\newblock {Untersuchungen {\"u}ber das logische Schlie{\ss}en I}.
\newblock \emph{Mathematische Zeitschrift}, 39:\penalty0 176--210, 1935.

\bibitem[Girard(1987)]{Girard}
J.-Y. Girard.
\newblock Linear logic.
\newblock \emph{Theor. Comput. Sci.}, 50\penalty0 (1):\penalty0 1--102, 1987.

\bibitem[Hodas and Miller(1991)]{hodas91lics}
J.~Hodas and D.~Miller.
\newblock {Logic programming in a fragment of intuitionistic linear logic.
  Extended abstract}.
\newblock In \emph{Proc. LICS '91}, pages 32--42, 1991.

\bibitem[Hodas and Miller(1994)]{hodas94ic}
J.~Hodas and D.~Miller.
\newblock {Logic programming in a fragment of intuitionistic linear logic}.
\newblock \emph{Information and Computation}, 110\penalty0 (2):\penalty0
  327--365, 1994.

\bibitem[Kanazawa(1992)]{KanazawaJoLLI}
M.~Kanazawa.
\newblock {The Lambek calculus enriched with additional connectives}.
\newblock \emph{J. Logic Lang. Inform.}, 1\penalty0 (2):\penalty0 141--171,
  1992.

\bibitem[Kanazawa(1999)]{Kanazawa}
M.~Kanazawa.
\newblock {Lambek calculus: Recognizing power and complexity}.
\newblock In \emph{JFAK. Essays dedicated to Johan van Benthem to the occasion
  of his 50th birthday}. Vossiuspers, Amsterdam Univ. Press, 1999.

\bibitem[Kanovich(1994)]{KanovichKazimierz}
M.~Kanovich.
\newblock {Horn fragments of non-commutative logics with additives are
  PSPACE-complete}.
\newblock {Extended abstract, presented at CSL '94, Kazimierz, Poland}, 1994.

\bibitem[Kanovich et~al.(2016{\natexlab{a}})Kanovich, Kuznetsov, and
  Scedrov]{KanKuzSceAPAL}
M.~Kanovich, S.~Kuznetsov, and A.~Scedrov.
\newblock {Reconciling Lambek's restriction, cut-elimination, and substitution
  in the presence of exponential modalities. ArXiv preprint 1608.02254},
  2016{\natexlab{a}}.

\bibitem[Kanovich et~al.(2016{\natexlab{b}})Kanovich, Kuznetsov, and
  Scedrov]{KanKuzSceFG}
M.~Kanovich, S.~Kuznetsov, and A.~Scedrov.
\newblock {Undecidability of the Lambek calculus with a relevant modality}.
\newblock In \emph{Proc. Formal Grammar '15 and '16}, volume 9804 of
  \emph{LNCS}, pages 240--256. Springer, 2016{\natexlab{b}}.

\bibitem[Kanovich et~al.(2016{\natexlab{c}})Kanovich, Kuznetsov, and
  Scedrov]{KanKuzSceLFCS}
M.~Kanovich, S.~Kuznetsov, and A.~Scedrov.
\newblock {On Lambek's restriction in the presence of exponential modalities}.
\newblock In S.~Artemov and A.~Nerode, editors, \emph{Proc. LFCS 2016}, pages
  146--158. Springer, 2016{\natexlab{c}}.

\bibitem[Kanovich et~al.(2017)Kanovich, Kuznetsov, and Scedrov]{KanKuzSceFCT}
M.~Kanovich, S.~Kuznetsov, and A.~Scedrov.
\newblock {Undecidability of the Lambek calculus with the Lambek calculus with
  subexponential and bracket modalities. ArXiv preprint 1608.04020. Accepted to
  FCT '17}, 2017.

\bibitem[Kanovich(1995)]{KanovichNeutral}
M.~I. Kanovich.
\newblock The complexity of neutrals in linear logic.
\newblock In \emph{Proc. LICS '95}, pages 486--495. IEEE, 1995.

\bibitem[Kuznetsov and Okhotin(2017)]{KuznetsovOkhotin}
S.~Kuznetsov and A.~Okhotin.
\newblock {Conjunctive categorial grammars}.
\newblock In \emph{Proc. MoL '17}, volume W17-3414 of \emph{ACL Anthology},
  pages 140--151, 2017.

\bibitem[Kuznetsov(2011)]{kuznetsov11}
S.~L. Kuznetsov.
\newblock {On the Lambek calculus with a unit and one division}.
\newblock \emph{Moscow Univ. Math. Bull.}, 66\penalty0 (4):\penalty0 173--175,
  2011.

\bibitem[Lambek(1958)]{Lambek58}
J.~Lambek.
\newblock The mathematics of sentence structure.
\newblock \emph{Amer. Math. Monthly}, 65:\penalty0 154--170, 1958.

\bibitem[Lambek(1961)]{Lambek61}
J.~Lambek.
\newblock {On the calculus of syntactic types}.
\newblock In \emph{{Structure of Language and Its Mathematical Aspects}},
  volume~12 of \emph{Proc. Symposia Appl. Math.}, pages 166--178. AMS, 1961.

\bibitem[Lambek(1969)]{Lambek69}
J.~Lambek.
\newblock {Deductive systems and categories II. Standard constructions and
  closed categories}.
\newblock In \emph{Category Theory, Homology Theory and Their Applications I},
  volume~86 of \emph{Lect. Notes Math.}, pages 76--122. Springer, 1969.

\bibitem[Lambek(1993)]{LambekZero}
J.~Lambek.
\newblock From categorial grammar to bilinear logic.
\newblock In \emph{{Substructural Logics}}, volume~2 of \emph{Studies in Logic
  and Computations}, pages 207--237. Clarendon Press, Oxford, 1993.

\bibitem[Lincoln et~al.(1992)Lincoln, Mitchell, Scedrov, and Shankar]{LMSS}
P.~Lincoln, J.~Mitchell, A.~Scedrov, and N.~Shankar.
\newblock Decision problems for propositional linear logic.
\newblock \emph{Annals of Pure and Applied Logic}, 56\penalty0 (1):\penalty0
  239--311, 1992.

\bibitem[Markov(1947)]{markov47dan}
A.~Markov.
\newblock On the impossibility of certain algorithms in the theory of
  associative systems.
\newblock \emph{Doklady Acad. Sci. USSR (N. S.)}, 55:\penalty0 583--586, 1947.

\bibitem[Miller(1994)]{miller94lics}
D.~Miller.
\newblock {A multiple-conclusion meta-logic}.
\newblock In \emph{Proc. LICS '94}, pages 272--281. IEEE, 1994.

\bibitem[Miller(1996)]{miller96tcs}
D.~Miller.
\newblock {Forum: A multiple-conclusion specification logic}.
\newblock \emph{Theor. Comput. Sci.}, 165\penalty0 (1):\penalty0 201--232,
  1996.

\bibitem[Moortgat(1997)]{MoortgatHandbook}
M.~Moortgat.
\newblock {Categorial type logics}.
\newblock In J.~van Benthem and A.~ter Meulen, editors, \emph{Handbook of Logic
  and Language}. Elsevier, 1997.

\bibitem[Moot and Retor{\'e}(2012)]{MootRetore}
R.~Moot and C.~Retor{\'e}.
\newblock \emph{{The logic of categorial grammars: a deductive account of
  natural language syntax and semantics}}, volume 6850 of \emph{LNCS}.
\newblock Springer, 2012.

\bibitem[Morrill(2017)]{MorrillPhilosophy}
G.~Morrill.
\newblock {Grammar logicised: relativisation}.
\newblock \emph{Linguistics and Philosophy}, 40\penalty0 (2):\penalty0
  119--163, 2017.

\bibitem[Morrill(2011)]{MorrillBook}
G.~V. Morrill.
\newblock \emph{{Categorial grammar: logical syntax, semantics, and
  processing}}.
\newblock Oxford Univ. Press, 2011.

\bibitem[Nigam(2012)]{nigam12lics}
V.~Nigam.
\newblock {On the complexity of linear authorization logics}.
\newblock In \emph{Proc. LICS '12}, pages 511--520. IEEE, 2012.

\bibitem[Nigam(2014)]{nigam14tcs}
V.~Nigam.
\newblock {A framework for linear authorization logics}.
\newblock \emph{Theor. Comput. Sci.}, 536\penalty0 (0):\penalty0 21--41, 2014.

\bibitem[Nigam and Miller(2009)]{nigam09ppdp}
V.~Nigam and D.~Miller.
\newblock {Algorithmic specifications in linear logic with subexponentials}.
\newblock In \emph{{Proc. PPDP '09}}, pages 129--140, 2009.

\bibitem[Nigam et~al.(2013)Nigam, Olarte, and Pimentel]{nigam13concur}
V.~Nigam, C.~Olarte, and E.~Pimentel.
\newblock {A general proof system for modalities in concurrent constraint
  programming}.
\newblock In \emph{CONCUR}, volume 8052 of \emph{LNCS}, pages 410--424.
  Springer, 2013.

\bibitem[Nigam et~al.(2016)Nigam, Pimentel, and Reis]{nigam16jlc}
V.~Nigam, E.~Pimentel, and G.~Reis.
\newblock {An extended framework for specifying and reasoning about proof
  systems}.
\newblock \emph{J. Logic Comput.}, 26\penalty0 (2):\penalty0 539--576, 2016.

\bibitem[Olarte et~al.(2015)Olarte, Pimentel, and Nigam]{olarte15tcs}
C.~Olarte, E.~Pimentel, and V.~Nigam.
\newblock {Subexponential concurrent constraint programming}.
\newblock \emph{Theor. Comput. Sci.}, 606:\penalty0 98--120, 2015.

\bibitem[Peirce(1885)]{Peirce}
C.~S. Peirce.
\newblock On the algebra of logic: a contribution to the philosophy of
  notation.
\newblock \emph{American Journal of Mathematics}, 7:\penalty0 180--202, 1885.

\bibitem[Pentus(1993)]{PentusCF}
M.~Pentus.
\newblock {Lambek grammars are context-free}.
\newblock In \emph{Proc. LICS '93}, pages 429--433. IEEE, 1993.

\bibitem[Pentus(1998)]{PentusFmonov}
M.~Pentus.
\newblock {Free monoid completeness of the Lambek calculus allowing empty
  premises}.
\newblock In \emph{Proc. Logic Colloquium '96}, volume~12 of \emph{Lect. Notes
  Logic}, pages 171--209. Springer, 1998.

\bibitem[Pentus(2006)]{PentusNP}
M.~Pentus.
\newblock {Lambek calculus is NP-complete}.
\newblock \emph{Theor. Comput. Sci.}, 357\penalty0 (1):\penalty0 186--201,
  2006.

\bibitem[Pfenning and Simmons(2009)]{DBLP:conf/lics/PfenningS09}
F.~Pfenning and R.~J. Simmons.
\newblock Substructural operational semantics as ordered logic programming.
\newblock In \emph{Proc. LICS '09}, pages 101--110. {IEEE} Computer Society,
  2009.

\bibitem[Post(1947)]{post47jsl}
E.~L. Post.
\newblock {Recursive unsolvability of a problem of Thue}.
\newblock \emph{J. Symb. Log.}, 12:\penalty0 1--11, 1947.

\bibitem[Savitch(1970)]{Savitch}
W.~J. Savitch.
\newblock {Relationships between nondeterministic and deterministic tape
  complexities}.
\newblock \emph{J. Comp. Syst. Sci.}, 4\penalty0 (2):\penalty0 177--192, 1970.

\bibitem[Schellinx(1991)]{Schellinx}
H.~Schellinx.
\newblock Some syntactical observations on linear logic.
\newblock \emph{J. Logic Computat.}, 1\penalty0 (4):\penalty0 537--559, 1991.

\bibitem[Thue(1914)]{Thue}
A.~Thue.
\newblock {Probleme {\"u}ber Ver{\"a}nderungen von Zeichenreihen nach gegebener
  Regeln}.
\newblock \emph{Kra. Vidensk. Selsk. Skrifter.}, 10, 1914.

\bibitem[Yetter(1990)]{Yetter}
D.~N. Yetter.
\newblock {Quantales and (noncommutative) linear logic}.
\newblock \emph{J. Symb. Logic}, 55\penalty0 (1):\penalty0 41--64, 1990.

\end{thebibliography}

\end{document}